\documentclass[12pt, a4paper, abstracton, bibliography=totoc]{scrartcl}
\pdfoutput=1

\usepackage{amsmath, amsthm, amsfonts, amssymb}
\usepackage[utf8]{inputenc}
\usepackage[dvipsnames]{xcolor}
\definecolor{darkgreen}{rgb}{0.0, 0.2, 0.13}
\usepackage{slashed} 
\usepackage[all, cmtip]{xy}
\usepackage{hyperref} 

\let\counterwithout\relax

\usepackage{chngcntr} 

\usepackage{graphicx} 
\usepackage[english]{babel}
\usepackage{microtype}
\usepackage{bm} 
\usepackage{xfrac} 
\usepackage{mathtools} 

\newcommand{\IN}{\mathbb{N}}
\newcommand{\IZ}{\mathbb{Z}}
\newcommand{\IQ}{\mathbb{Q}}
\newcommand{\IR}{\mathbb{R}}
\newcommand{\IC}{\mathbb{C}}

\newcommand{\IB}{\mathfrak{B}}

\newcommand{\id}{\mathrm{id}}

\newcommand{\ch}{\operatorname{ch}}

\newcommand{\op}{\mathrm{op}}

\newcommand{\KK}{K\! K}

\newcommand{\alg}{\mathrm{alg}}

\newcommand{\CstarAlg}{\mathbf{C^*\text{-}Alg}}
\newcommand{\starAlg}{\mathbf{{}^*\text{-}Alg}}
\newcommand{\Ad}{\operatorname{Ad}}
\newcommand*{\into}{\hookrightarrow}
\newcommand*{\onto}{\twoheadrightarrow}
\newcommand*{\Top}{\mathrm{top}}

\newcommand{\RNum}[1]{\uppercase\expandafter{\romannumeral #1\relax}}

\newtheorem{thm}{Theorem}[section]
\newtheorem*{thm*}{Theorem}
\newtheorem{cor}[thm]{Corollary}
\newtheorem*{cor*}{Corollary}
\newtheorem{lem}[thm]{Lemma}
\newtheorem{prop}[thm]{Proposition}

\newtheorem{Property}[thm]{Property}

\newtheorem{thmintro}{Theorem}

\theoremstyle{definition}
\newtheorem{rem-alt}[thm]{Remark}
\newtheorem*{rem*}{Remark}
\newtheorem{ex-alt}[thm]{Example}

\newtheorem*{example*}{Example}
\newtheorem*{examples*}{Examples}
\newtheorem{defn-alt}[thm]{Definition}
\newtheorem{nota-alt}[thm]{Notation}
\newtheorem{properties-alt}[thm]{Properties}

\newenvironment{defn}    
{%
	\pushQED{\qed}\begin{defn-alt}}
	{\popQED\end{defn-alt}}
	
{%
	\pushQED{\qed}\begin{ex-alt}}
	{\popQED\end{ex-alt}}

\newenvironment{rem}    
{%
	\pushQED{\qed}\begin{rem-alt}}
	{\popQED\end{rem-alt}}

{%
	\pushQED{\qed}\begin{nota-alt}}
	{\popQED\end{nota-alt}}
	
\newenvironment{properties}    
{%
	\pushQED{\qed}\begin{properties-alt}}
	{\popQED\end{properties-alt}}

\numberwithin{equation}{section}
\counterwithout{footnote}{section}

\makeatletter
\def\blfootnote{\gdef\@thefnmark{}\@footnotetext}
\makeatother

\begin{document}

\title{$\mathclap{\text{Strong Novikov conjecture for low degree}}$\\ $\mathclap{\text{cohomology and exotic group C*-algebras}}$}

\author{
Paolo Antonini\thanks{
Dipartimento di Matematica e Fisica ``E.\ de Giorgi", Universit\`a del Salento, Via per Arnesano, 73100 Lecce, Italy\newline
\href{mailto:paolo.antonini@unisalento.it}{paolo.antonini@unisalento.it}}
\and
Alcides Buss\thanks{
\mbox{Departamento de Matem\'atica,
 Universidade Federal de Santa Catarina,
 88.040-900 Florian\'opolis,
 Brazil}\newline
 \href{mailto:alcides.buss@ufsc.br}{alcides.buss@ufsc.br}
}
\and
Alexander Engel\thanks{Fakult{\"a}t f{\"u}r Mathematik,
Universit{\"a}t Regensburg,
93040 Regensburg,
Germany\newline
\href{mailto:alexander.engel@mathematik.uni-regensburg.de}{alexander.engel@mathematik.uni-regensburg.de}}
\and
Timo Siebenand\thanks{Mathematisches Institut, Universit{\"a}t M{\"u}nster, 48149 M{\"u}nster, Germany \newline{\href{mailto:timo.siebenand@uni-muenster.de}{timo.siebenand@uni-muenster.de}}}
}

\date{}

\maketitle

\vspace*{-6ex}
\begin{abstract}\noindent
We strengthen a result of Hanke--Schick about the strong Novikov conjecture for low degree cohomology by showing that their non-vanishing result for the maximal group C*-algebra holds for many other exotic group C*-algebras, in particular the one associated to the smallest strongly Morita compatible and exact crossed product functor used in the new version of the Baum--Connes conjecture.
To achieve this we provide a Fell absorption principle for certain exotic crossed product functors.
\end{abstract}

\tableofcontents

\section{Introduction}
Recall the following result of Hanke and Schick \cite{hanke_schick}. Let $G$ be a discrete group and denote by $\Lambda^\ast(G) \subset H^\ast(BG;\IQ)$ the subring of the singular cohomology theory $H^\ast(BG;\IQ)$ with rational coefficients generated by $H^{\le 2}(BG;\IQ)$, the rational cohomology classes of degree at most two.
Further, let $\ch\colon K_\ast(BG)\to
H_\ast(BG;\IQ)$ be the homological Chern character from the $K$-homology to the homology
of the classifying space $BG$ of $G$.
\begin{thmintro}[{\cite{hanke_schick}}]\label{hanke_schick_intro}
Let $h \in K_*(BG)$ such that there is $c \in \Lambda^\ast(G)$ with $\langle c,\ch(h)\rangle \not= 0$.

Then $h$ is not mapped to zero under the assembly map $K_*(BG) \to K_*(C^\ast_{\max} G) \otimes \IR$.
\end{thmintro}

Because the rational injectivity of the assembly map is known as the strong Novikov conjecture, their result states that the strong Novikov conjecture is true for those $K$-homology classes which can be detected by low degree cohomology classes.

Note that the strong Novikov conjecture firstly implies the classical Novikov conjecture about homotopy invariance of higher signatures, and secondly provides obstructions (the higher $\hat{A}$-genera) to the existence of positive scalar curvature metrics on manifolds. 
The Novikov conjecture for low degree cohomology classes was proven with different methods by Connes, Gromov and Moscovici \cite{MR1204787} and by Mathai \cite{MR1998926}.

If $G$ is discrete and torsion free, the Baum--Connes conjecture states that the analytic assembly map $K_*(BG) \to K_*(C^\ast_r G)$ is an isomorphism. Note importantly, that on the right hand side we use the reduced group C*-algebra of $G$. Therefore the Baum--Connes conjecture, which is known for many groups, predicts that in the above theorem of Hanke and Schick we should be able to put the reduced, instead of maximal, group C*-algebra on the right hand side. 
This idea is supported by the results in \cite{MR3419768,AAS3} where it is shown that the image of the analytic assembly map is related to the $\tau$-part of the $K$-theory of the reduced group C*-algebra. On the other hand, the $\tau$-part of the $K$-theory for the reduced and the maximal group C*-algebras are canonically identified.

Recently a new version of the Baum--Connes conjecture (with coefficients) was formulated in \cite{BGW} using a new crossed product, the so-called minimal exact and (strongly) Morita compatible crossed product functor. This is a functor from the category of $G$-actions on $C^*$-algebras $A$ that assigns to each such $C^*$-algebra a crossed product $A\rtimes_\epsilon G$ lying between the maximal and reduced crossed products. This functor, moreover, preserves exact sequences and Morita equivalences. The Baum--Connes conjecture with coefficients then states that a certain assembly map
$$K^\Top_*(G;A)\to K_*(A\rtimes_\epsilon G)$$
is an isomorphism. The original Baum--Connes conjecture with coefficients states that a similar assembly map is an isomorphism for the reduced crossed product $A\rtimes_r G$ in place of $A\rtimes_\epsilon G$. Unfortunately this original version is false in general due to exactness obstructions \cite{MR1911663}, as opposed to the new one where no counter-examples are known.

The functor $A\mapsto A\rtimes_\epsilon G$ applied to $A=\IC$ gives in particular a group $C^*$-algebra $C^*_\epsilon(G)$ that should play an important role. If the new Baum--Connes conjecture is true for $G$, then it yields an isomorphism $K_*(BG)\to K_*(C^*_\epsilon(G))$ whenever $G$ is a torsion-free discrete group.

All this indicates that a version of Theorem~\ref{hanke_schick_intro} should hold with $C^*_\epsilon(G)$ in place of the full group $C^*$-algebra $C^*_{\max}(G)$. The main goal of this paper is to confirm that this is in fact true. To be more precise, let $G$ be a finitely presented group and, as before, denote by $\Lambda^\ast(G) \subset H^\ast(BG;\IQ)$ the subring generated by $H^{\le 2}(BG;\IQ)$, the cohomology classes of degree at most two.

\begin{thmintro}\label{main_thm_intro}
Let $h \in K_*(BG)$ such that there is $c \in \Lambda^\ast(G)$ with $\langle c,\ch(h)\rangle \not= 0$.

Then $h$ is not mapped to zero under the assembly map $K_*(BG) \to K_*(C^\ast_\epsilon G) \otimes \IR$.
\end{thmintro}

In order to prove Theorem~\ref{main_thm_intro} we will use the machinery of exotic crossed products developed in several papers, \cite{BGW,bew2,bew1,bew3}.  Notice that $C^*_\epsilon(G)=C^*_r(G)$ if $G$ is exact. It is plausible that this is indeed true for every group $G$ as was stated in \cite{bew3}. Unfortunately there is a gap in one of the proofs of that statement and as a consequence this was left open (see the erratum in the appendix of arXiv version~3 of \cite{bew3}). On the other hand, as already remarked in \cite{BGW}, the group $C^*$-algebra $C^*_\epsilon(G)$ is never equal to $C^*_{\max}(G)$ unless $G$ is amenable. In this sense our Theorem~\ref{main_thm_intro} really improves Hanke--Schick's Theorem~\ref{hanke_schick_intro}.

We have written the proof of Theorem~\ref{main_thm_intro} in such a way that it may be read without first reading Section~\ref{sec_exotic_crossed_prod} on exotic crossed products. Moreover, our proof applies to the group algebra $C^*_\mu(G)$ of every exact correspondence crossed product functor $A\mapsto A\rtimes_\mu G$. Therefore Theorem~\ref{main_thm_intro} holds for every exotic group $C^*$-algebra $C^*_\mu(G)$ above $C^*_\epsilon(G)$, that is, one for which the identity on $G$ extends to a surjection $C^*_\mu(G)\to C^*_\epsilon(G)$. Everything one needs to know about the kind of exotic crossed products we are using in the proof of Theorem~\ref{main_thm_intro} is summarized in Properties~\ref{properties_crossed_product}.

The main technical result that we will prove in Section~\ref{sec_exotic_crossed_prod} is a Fell absorption principle for exotic crossed products (Lemma~\ref{lem_strong_fell}). A consequence of it is the following result\footnote{We will prove this result for all locally compact groups, but here we only state it for discrete groups for simplicity.}:

\begin{thmintro}\label{main_thm_intro_fell}
Let $-\rtimes_\mu \, G$ be a correspondence crossed product functor.

Then the coproduct $\Delta\colon \IC G \to \IC G \odot \IC G$, $\sum a_g g \mapsto \sum a_g (g \otimes g)$ extends continuously to a $^*$-homomorphism
\[\Delta\colon C_\mu^* G \to C_{\max}^* G \otimes_\mu C_\mu^* G.\]
\end{thmintro}

It was already known that the coproduct extends to $C^*_\mu(G)\to C^*_{\max}(G)\otimes_{\min} C^*_\mu(G)$ (see \cite[Cor.~3.13]{KLQ-Exotic}). But the fact that we are able to lift it to a map where we use a different tensor product $-\otimes_\mu C^*_\mu(G)$ is a crucial ingredient in our proof of Theorem~\ref{main_thm_intro} as this tensor product enjoys properties from the crossed product functor $-\rtimes_\mu G$.
In particular, this applies to the group algebra $C^*_\epsilon(G)$ of the minimal exact and strongly Morita compatible crossed product functor $A\mapsto A\rtimes_\epsilon G$. The fact that this functor is exact is actually equivalent to the exactness of the associated tensor product functor $A\mapsto A\otimes_\epsilon C^*_\epsilon(G)$.

\subsection*{Acknowledgements} We thank Sara Azzali, Siegfried Echterhoff and Rufus Willett for helpful discussions. We also thank the anonymous referee for his or her comments.

P.A.\ wishes to thank the International School for Advanced Studies, SISSA where he held a postdoctoral position while the paper was written.

A.B.\ is supported by Capes-Humboldt and CNPq - Brazil.

A.E.\ acknowledges support by the Priority Programme SPP 2026 \emph{Geometry at Infinity} (EN 1163/3-1, \emph{Duality and the coarse assembly map}) and the SFB 1085 \emph{Higher Invariants}, both funded by the DFG.

T.S.\ is supported by the DFG under Germany's Excellence Strategy - EXC 2044 - 390685587, Mathematics Münster: Dynamics – Geometry - Structure and by SFB 878.

\section{Exotic crossed products}\label{sec_exotic_crossed_prod}
In this section we describe and prove some properties of exotic crossed products. Even though, in applications to the Novikov conjecture only discrete groups are involved, we present these facts in full generality. 
Let $G$ be a locally compact group with a fixed (left invariant) Haar measure (which we simply write as $ds$ in integrals) and let $\Delta$
the associated modular function on $G$.
\begin{itemize}
\item We let $\CstarAlg(G)$ (resp. $\mathbf{CP}(G)$)
denote the category
of $G$-C*-algebras with $G$-equivariant *-homomorphisms (resp., $G$-equivariant
completely positive maps) as morphisms.
\item For $\CstarAlg(\{e\})$ (resp. $\mathbf{CP}(\{e\})$), where $e$ denotes the trivial group, we also
write $\CstarAlg$ (resp. $\mathbf{CP}$).
\item Further, let $\starAlg$ denote the
category of involutive algebras with *-homomorphisms as morphisms.
\end{itemize}

\pagebreak[2]
\begin{defn}[Crossed products]\mbox{}
\begin{enumerate}
\item By $-\rtimes_{\mathrm{alg}}G: \CstarAlg(G) \to \starAlg$ we denote the functor
mapping a $G$-C*-algebra $A$ with action $\alpha$ to
\begin{align*}
	A\rtimes_{\mathrm{alg}}G \coloneqq C_c(G,A)
\end{align*}
as a vector space equipped with the product
\begin{align*}
	(f\ast g)(t) \coloneqq \int f(s) \alpha_s(g(s^{-1}t))ds
\end{align*}
for $f,g\in A\rtimes_{\mathrm{alg}}G$ and $t\in G$, and the involution
\begin{align*}
	f^*(s) \coloneqq \Delta(s^{-1}) \alpha_s((f(s^{-1}))^*)
\end{align*}
for $f\in A\rtimes_{\mathrm{alg}}G$ and $s\in G$, and mapping a $G$-equivariant
*-homomorphism $\varphi\colon A \to B$ to the *-homo\-mor\-phism
\begin{align*}
	\varphi\rtimes_{\mathrm{alg}}G\colon A\rtimes_{\mathrm{alg}}G \to B\rtimes_{\mathrm{alg}}G, \, f \mapsto \varphi\circ f\,.
\end{align*}
\item By $-\rtimes_{\max} G\colon \CstarAlg(G) \to \CstarAlg$ we
denote the maximal crossed product functor which comes with a natural transformation
\[
(\kappa_A\colon A\rtimes_{\mathrm{alg}}G\to A\rtimes_{\max} G)_{A\in \CstarAlg(G)}
\]
consisting
of injective *-homomorphisms with dense image.
\item Finally, let $-\rtimes_r G \colon \CstarAlg(G) \to \CstarAlg$ be the
reduced crossed product functor.

There is a natural transformation
$\Lambda\colon -\rtimes_{\max} G \to -\rtimes_r G$ between the functors $-\rtimes_{\max} G$ and
$-\rtimes_r G$ consisting of surjective *-homomorphisms such that
$\Lambda\circ \kappa$ consists of injective *-homomorphisms (with dense image).\qedhere
\end{enumerate}
\end{defn}

A \emph{crossed product functor} $-\rtimes_\mu G$ is a functor
\begin{align*}
	-\rtimes_\mu G \colon \CstarAlg(G) \to \CstarAlg
\end{align*}
together with natural transformations
\[q\colon-\rtimes_{\max} G \to -\rtimes_\mu G \quad \text{ and } \quad s\colon-\rtimes_\mu G \to -\rtimes_r G\]
consisting of surjective *-homomorphisms such that
$s\circ q = \Lambda$. In particular, we have that $q\circ \kappa:-\rtimes_{\mathrm{alg}}G \to -\rtimes_\mu G$ is a natural assigment consisting
of injective *-homomorphisms with dense image. Therefore we can consider $A\rtimes_{\alg}G$ as
a *-subalgebra of $A\rtimes_\mu G$ for all $G$-C*-algebras~$A$.

We write $\mathcal{M}(A)$ for the multiplier algebra of a C*-algebra $A$. 
Let
$\varphi \colon A\to \mathcal{M}(B)$ be a nondegenerate *-homomorphism between a C*-algebra
$A$ and the multiplier algebra $\mathcal{M}(B)$ of a C*-algebra $B$. This means that $\varphi(A)B$ is dense in $B$. In this case there is a unique extension of $\varphi$ to a homomorphism on $\mathcal{M}(A)$; we denote it by $\overline{\varphi}:\mathcal{M}(A)\to \mathcal{M}(B)$.

If $-\rtimes_\mu G$ is a crossed product functor, then we typically write
$C_\mu^* G$ for $\mathbb{C}\rtimes_\mu G$. Furthermore, let $A$ be a $G$-C*-algebra and
let $(A\rtimes_{\max} G,\iota_A,\iota_G)$ be the maximal crossed product of $A$ together
with the universal covariant representation $\iota_A \colon A\to \mathcal{M}(A\rtimes_{\max} G)$
and $\iota_G \colon G \to \mathcal{M}(A\rtimes_{\max} G)$. Then
\begin{align*}
	\iota_{A,\mu}& \coloneqq \overline{q_A} \circ \iota_A \colon A \to \mathcal{M}(A\rtimes_\mu G)\\
	\iota_{G,\mu}& \coloneqq \overline{q_A}\circ \iota_G \colon G \to \mathcal{U}\mathcal{M}(A\rtimes_\mu G)
\end{align*}
is a covariant representation of $A$.

\subsection{Properties of correspondence crossed product functors}

\begin{lem}\label{lem-directsums}
Let $-\rtimes_\mu G$ be a crossed product functor.
Then $-\rtimes_\mu G$ preserves direct sums of C*-algebras. To be more precise, if $\{A_i\colon i\in I\}$ is any collection of $G$-C*-algebras,
then there is a canonical isomorphism
$$\big(\bigoplus_{i\in I} A_i\big)\rtimes_{\mu}G\cong \bigoplus_{i\in I} (A_i\rtimes_\mu G)\,.$$
\end{lem}
\begin{proof}
Let $A$ be the direct sum $\bigoplus_i A_i$. This is the universal $G$-C*-algebra generated by orthogonal
 copies of $A_i$ as ideals. The $G$-equivariant inclusion $\phi_i\colon A_i\to A$ then lifts to a *-homomorphism $\phi_i\rtimes_\mu G\colon A_i\rtimes_\mu G\to A\rtimes_\mu G$ and the images of 
 these maps are mutually orthogonal, so we get a well-defined *-homomorphism $\phi\colon \bigoplus_i (A_i\rtimes_\mu G)\to A\rtimes_\mu G$, which is clearly also surjective. It is also injective because the canonical $G$-equivariant projections 
$\psi_i\colon A\to A_i$ by functoriality yield *-homomorphisms
 $\psi_i\rtimes_\mu G\colon A\rtimes_\mu G\to A_i\rtimes_\mu G$ that therefore give a 
 *-homomorphism $\psi\colon A\rtimes_\mu G\to \prod_i (A_i\rtimes_\mu G)$ such that $\psi\circ\phi$ equals the canonical embedding $\bigoplus_i (A_i\rtimes_\mu G)\to \prod_i (A_i\rtimes_\mu G)$.
\end{proof}

We are going to use a class of crossed products which is well behaved with respect to completely positive maps. These are proven in \cite{bew1} to be exactly the \emph{correspondence crossed products}, i.e. the crossed products which are functorial for 
correspondences defined as bimodules in the sense of Kasparov. Indeed they allow for the construction of a descent morphism in equivariant $K\!K$-theory. Among several equivalent definitions (cf.~\cite[Thm.~4.9]{bew1}) we recall the one related to completely postive maps.

\begin{defn}
	A crossed product functor $-\rtimes_\mu G$ has the \emph{cp-map property}
	(or equivalent, is a correspondence crossed product functor), if $-\rtimes_\mu G$ extends 
	to a functor
	\[-\rtimes_\mu G\colon\mathbf{CP}(G) \to \mathbf{CP}\]
	in the following sense:
	
	For all $G$-C*-algebras $A$ and $B$ and every $G$-equivariant completely positive map
	$\varphi\colon A \to B$, there is a completely positive map $\varphi\rtimes_\mu G\colon
	A\rtimes_\mu G \to B\rtimes_\mu G$ 
	determined by 
	 $(\varphi\rtimes_\mu G)(f) = \varphi\circ f$
	for all $f\in A\rtimes_{\alg}G$.
\end{defn}

The next lemma is a direct combination of Theorem 4.9 and Lemma 3.3 in \cite{bew1} (taking into account the implication ``hereditary subalgebra property $\Rightarrow$ ideal property'').
\begin{lem}\label{pre:propCCP}
	Every correspondence crossed product functor $-\rtimes_\mu\, G$ is
	 \emph{functorial for generalised homomorphisms}. In other words, for all
	$G$-C*-algebras $A$ and $B$ and every $G$-equivariant *-homomorphism
	$\phi: A \to \mathcal{M}(B)$, there exists a *-homomorphism
	\[\phi\rtimes_\mu G\colon A\rtimes_\mu G \to \mathcal{M}(B\rtimes_\mu G)\] given by
	\[(\phi\rtimes_\mu G) (f) g = (\phi \circ f)\ast g\]
	for all $f\in A\rtimes_{alg}G$
	and $g\in B\rtimes_{\alg} G$.
\end{lem}

\begin{rem}\label{canonicalmorphism}
	Note that for a crossed product functor $-\rtimes_\mu G$ which is functorial for generalised 
	homomorphisms,
	the unitary group representation $\iota_{G,\mu}\colon G \to \mathcal{U}\mathcal{M}(A\rtimes_\mu G)$ integrates
	to a *-homomorphism $\iota_{G,\mu}\colon C_\mu^* G  \to \mathcal{M}(A\rtimes_\mu G)$
	for all $G$-C*-algebras $A$. 

To see this, consider the $G$-equivariant unital *-homomorphism $m\colon\mathbb{C}\to \mathcal{M}(A)$ (which is just the scalar multiplication with the unit $1_A$); since $-\rtimes_\mu G$ is functorial for
generalised homomorphisms, $m$ induces a *-homomorphism $m\rtimes_\mu G\colon C_\mu^* G \to 
\mathcal{M}(A\rtimes_\mu G)$, which is the integrated form
of $\iota_{G,\mu}$.
\end{rem}
Furthermore, a consequence of Lemma~\ref{pre:propCCP} is the following:

\begin{prop}\label{pre:injectiveIndcoRep}
	Let $G$ be a discrete group, $A$ a $G$-C*-algebra and $(A\rtimes_{\max} G, \iota_A,\iota_G)$
	the maximal crossed product together with the universal covariant representation
	\[\iota_A\colon A \to \mathcal{M}(A\rtimes_{\max} G) \quad \text{ and } \quad \iota_G\colon G \to \mathcal{U}\mathcal{M}(A\rtimes_{\max} G)\,.\]
		
	Then $(A\rtimes_{\max} G, \operatorname{Ad}(\iota_G))$ is a $G$-C*-algebra and $\iota_A$ factors through a $G$-equivariant *-homomorphism $A \to A\rtimes_{\max} G$, where the $G$-action on $A\rtimes_{\max} G$ is now given by $\operatorname{Ad}(\iota_G)$.
	
Further, if $-\rtimes_\mu\, G$ is a correspondence crossed product functor, then the induced *-homomorphism 
	\[\iota_A\rtimes_\mu G\colon A\rtimes_\mu G \to (A\rtimes_{\max} G)\rtimes_{\mu,\Ad(\iota_G)}G\]
	is injective.
\end{prop}
\begin{proof}
Notice that, since $G$ is a discrete group, the image of $\iota_A$ is contained in $A\rtimes_{\mathrm{alg}}G\subseteq A\rtimes_{\max}G$; indeed, it is given by $\iota_A(a) = a\delta_e$ for all $a\in A$ where $e\in G$ is the identity element.

	Let $\mathrm{E}\colon A\rtimes_{\max} G \to A$ be the canonical conditional expectation; this is the continuous extension of the evaluation $A\rtimes_{\mathrm{alg}}G\to A$ at the identity element $e\in G$. It is a $G$-equivariant completely positive map. Since 
	$\mathrm{E}\circ \iota_A = \mathrm{id}_A$, we obtain by the cp-map property
	\[\mathrm{id}_{A\rtimes_\mu G} =  (\mathrm{E}\rtimes_\mu G) \circ (\iota_A\rtimes_\mu G)\,,\]
	which proves the injectivity of $\iota_A\rtimes_\mu G$.
\end{proof}
Let $(A,\alpha)$ be a $G$-C*-algebra
with a $G$-action $\alpha$. Recall that the action $\alpha$ is called \emph{unitarily implemented}
if there exists a strictly continuous unitary group representation
$u\colon G\to \mathcal{M}(A)$ such that
\begin{align*}
	\alpha_s = \operatorname{Ad}(u_s) 
\end{align*}
for all $s\in G$. In the language of \cite[Sec.\ 5]{bew1} this means that $\alpha$ is exterior equivalent to the trivial action.

The next proposition follows from \cite[Lem.~5.2]{bew1} and \cite[Thm.~4.9]{bew1}.

\begin{prop}\label{CF:extEq}
		Let $(A,\alpha)$ be a $G$-C*-algebra with a unitarily implemented action $\alpha$, and denote by $u\colon G \to 
		\mathcal{U}\mathcal{M}(A)$ the corresponding unitary group representation.
		 
		If $-\rtimes_\mu\, G$ is a correspondence crossed product functor, then 
		\begin{align*}
			\varphi\colon A\rtimes_{\alg,\,\alpha}G \to A\rtimes_{\alg,\, \operatorname{id}}G\,,\quad 
			f \mapsto (G\to A,\, s\mapsto f(s)u_s)
		\end{align*}
		extends to a *-isomorphism
		$\varphi\colon A\rtimes_{\mu,\alpha} G \to A\rtimes_{\mu,\,\operatorname{id}} G$.
	\end{prop}
	
\subsection{Fell absorption principle for crossed product functors}
	Let $-\rtimes_\mu\, G$ be a correspondence crossed product functor. First of all, we define
	a C*-tensor product 
	\begin{align*}
		-\otimes_\mu C_\mu^* G \colon \CstarAlg\to \CstarAlg
	\end{align*}
	with the group C*-algebra $C_\mu^* G$.
	For this purpose, we let $A$ be a C*-algebra. Then $A$ can be considered as a $G$-C*--algebra with
	$G$ acting trivially on $A$. We set
	\[A\otimes_\mu C_\mu^* G  \coloneqq A\rtimes_{\mu,\,\operatorname{id}} G.\]
The following proposition proves that this is indeed a C*-tensor product of $A$ and $C_\mu^* G$ and justifies our notation. 

	We first recall some notation. Let $A$ be a C*-algebra and
		let $(A\rtimes_{\max} G, \iota_A, \iota_G)$ be the maximal crossed product
		together with the universal covariant representation
		\[\iota_A \colon A\to \mathcal{M}(A\rtimes_{\max} G) \quad \text{ and } \quad \iota_G \colon G \to \mathcal{U}\mathcal{M}(A\rtimes_{\max} G)\,.\]
		Let
		$q_{A}\colon A\rtimes_{\max} G \to A\rtimes_\mu G$ be the natural quotient map
		and $ \iota_{A,\mu} \coloneqq \overline{q_A}\circ \iota_A$. Finally,
		let $\iota_{G,\mu}\colon C_\mu^* G \to \mathcal{M}(A\rtimes_\mu\, G)$ be the canonical 
		*-homomorphism described in Remark \ref{canonicalmorphism}.

	\begin{prop}\label{prop_is_tensor_product}
	Let $G$ act trivially on $A$.
		Then the *-homomorphisms
		$\iota_{A,\mu}$ and $\iota_{G,\mu}$ commute and we have
		\[\iota_{A,\mu}(a) \cdot \iota_{G,\mu}(f)
		\in A\rtimes_\mu G\]
		for every $a\in A$ and $f\in \mathbb{C}\rtimes_{\alg}G$. The induced *-homomorphism
		\begin{align*}
			\iota_{A,\mu}\odot \iota_{G,\mu}\colon A\odot C_\mu^* G  \to A\rtimes_\mu G
		\end{align*}
		is  injective and has dense range.
		In particular, $A\rtimes_\mu G$ is a tensor product for $A$ and $C_\mu^* G$.
	\end{prop}
	\begin{proof}
	For $a\in A$ and $f\in \mathbb{C}\rtimes_{\mathrm{alg}}G$ we have
		$\iota_{A,\mu}(a) \cdot \iota_{G,\mu}(f) = (a\otimes f \colon G\to A,\, s\mapsto f(s)a)$, which is contained in $A\rtimes_{\mathrm{alg}}G \subseteq A\rtimes_{\mu}G$, and it is clear that the *-homomorphisms
		$\iota_{A,\mu}$ and $\iota_{G,\mu}$ commute. We also obtain
		$\iota_{A,\mu}(a) \cdot \iota_{G,\mu}(x) \in A\rtimes_\mu G$ for all $a\in A$ and $x\in C_\mu^* G$. It is also quickly verified that 
		$\mathrm{span}\lbrace \iota_{A,\mu}(a) \cdot \iota_{G,\mu}(f)\colon a\in A,\, f\in \mathbb{C}\rtimes_{\mathrm{alg}}G \rbrace$ is a dense subspace of
		 $A\rtimes_\mu G$.
		
		Therefore it remains to prove  that $\iota_{A,\mu}\odot \iota_{G,\mu}$ is injective.
		For this purpose, assume that $x\in A\odot C_\mu^* G$ is an element in the kernel of
		$\iota_{A,\mu}\odot \iota_{G,\mu}$. Then there are elements
		$a_1,\ldots,a_n \in A$ and $b_1,\ldots, b_n \in C_\mu^*(G)$ such that
		$x= \sum_i a_i \otimes b_i$ and $\lbrace b_i \colon 1\leq i \leq n \rbrace$
		is linearly independent. Since $-\rtimes_\mu G$ has the cp-map property,
		every state $\varphi\in \mathcal{S}(A)$ on $A$ induces a completely
		positive map $\varphi\rtimes_\mu G\colon A\rtimes_\mu G \to C_\mu^* G$ given by
		\[(\varphi\rtimes_\mu G)(f)(t) = \varphi(f(t))\]
		for $f\in A\rtimes_{\mathrm{alg}}G$ and $t\in G$.
		We obtain $$(\varphi\rtimes_\mu G)(\iota_{A,\mu}\odot \iota_{G,\mu}(x))
		= \sum_i \varphi(a_i) b_i=0$$ for all states $\varphi \in \mathcal{S}(A)$ on $A$.
		This implies $\varphi(a_i) = 0$ for all $\varphi \in \mathcal{S}(A)$ and
		$1\leq i \leq n$ and finally $a_i = 0$ for all $1\leq i \leq n$. 
		This completes the proof.
	\end{proof}
	\begin{rem}
		In the following we will identify $A\odot C_\mu^*(G)$  with a *-subalgebra
		of $A\rtimes_\mu G$ (via $\iota_{A,\mu}\odot \iota_{G,\mu}$) whenever $A$ is a C*-algebra
		(with trivial $G$-action).
		
		Note that the tensor product $-\otimes_\mu C_\mu^* G$ is just the restriction of the
		original crossed product functor $-\rtimes_\mu G$ to C*-algebras with the trivial $G$-action and hence inherits properties like (generalised) functoriality, exactness, etc.
	\end{rem}

Let $G$ be a locally compact group, $-\rtimes_\mu \, G$ be a correspondence crossed product functor and $A$ be a $G$-C*-algebra.
		Let $(A\rtimes_{\max} G,\iota_A, \iota_G)$ be the maximal crossed product of $A$
		and 
		$u_{G,\mu}\colon G \to \mathcal{U}\mathcal{M}(C_\mu^* G)$ the canonical unitary representation
		on $C^*_\mu G$.
			
\begin{lem}\label{lem_strong_fell}
		The *-homomorphism
		\begin{itemize}
		\item $\iota_A \otimes_\mu 1\colon A \to \mathcal{M}((A\rtimes_{\max} G) \otimes_\mu C_\mu^* G)$
		induced by
		\begin{align*}
			(\iota_A \otimes_\mu 1)(a) (b\otimes c) \coloneqq \iota_A(a)b \otimes c
		\end{align*}
		for $a\in A$, $b\in A\rtimes_{\max} G$ and $c\in C^*_\mu G$,
		and the group homomorphism
		\item $\iota_G \otimes_\mu u_{G,\mu}\colon G \to \mathcal{U}\mathcal{M}((A\rtimes_{\max} G) \otimes_\mu C_\mu^* G)$
		induced by
		\begin{align*}
			(\iota_G \otimes_\mu u_{G,\mu})(s) (b\otimes c) = \iota_G(s)b \otimes u_{G,\mu}(s) c
		\end{align*}
		for all $s\in G$, $b\in A\rtimes_{\max} G$ and $c\in C_\mu^* G$
		\end{itemize}
		exist and define a nondegenerate covariant representation of $A$. The integrated form
		\[(\iota_A \otimes_\mu 1) \rtimes_{\max} (\iota_G \otimes_\mu u_{G,\mu}) \colon A\rtimes_{\max} G \to \mathcal{M}((A\rtimes_{\max} G) \otimes_\mu C_\mu^* G )\]
		factors through
		$q_A \colon A\rtimes_{\max} G \to A\rtimes_\mu G$, i.e., there is a unique *-homomorphism
		\[(\iota_A \otimes_\mu 1) \rtimes_\mu (\iota_G \otimes_\mu u_{G,\mu})\colon A\rtimes_\mu G
		\to \mathcal{M}((A\rtimes_{\max} G) \otimes_\mu C_\mu^* G)\]
		such that $(\iota_A \otimes_\mu 1) \rtimes_{\max} (\iota_G \otimes_\mu u_{G,\mu})
		= ((\iota_A \otimes_\mu 1) \rtimes_\mu (\iota_G \otimes_\mu u_{G,\mu})) \circ q_A$.
		If $G$ is discrete, then $(\iota_A \otimes_\mu 1) \rtimes_\mu (\iota_G \otimes_\mu u_{G,\mu})$
		is injective.
	\end{lem}
	\begin{proof}
		Let us first show existence of $\iota_A\otimes_\mu 1$ and $\iota_G \otimes_\mu 
		u_{G,\mu}$. It is known \cite[Prop.~B.21]{raeburn_williams} that there are *-mono\-mor\-phisms
		\begin{alignat*}{3}	
		j_{A\rtimes_{\max} G} & \colon & \mathcal{M}(A\rtimes_{\max} G) & \longrightarrow \mathcal{M}((A\rtimes_{\max} G)\otimes_\mu C_\mu^* G)\\
		j_{C^*_\mu G} & \colon & \mathcal{M}(C^*_\mu G) & \longrightarrow \mathcal{M}((A\rtimes_{\max} G) \otimes_\mu C_\mu^* G)
		\end{alignat*}
		which satisfy
		\begin{align*}
		j_{A\rtimes_{\max} G}(m) (a\otimes b) & = ma \otimes b\\
		j_{C_\mu^* G}(n)(a\otimes b) & = a\otimes nb\\
		j_{A\rtimes_{\max} G}(a) \cdot j_{C_\mu^* G}(b)
		& = a\otimes b
		\end{align*}
		for all $m\in \mathcal{M}(A\rtimes_{\max} G)$, $n\in \mathcal{M}(C_\mu^* G)$,
		$a\in A\rtimes_{\max} G$ and $b\in C_\mu^* G$. Note that by the first two equations in the previous display the images of $j_{A\rtimes_{\max} G}$ and $j_{C^*_\mu G}$ commute with each other. We now set
		\[\iota_A\otimes_\mu 1 \coloneqq j_{A\rtimes_{\max} G}\circ \iota_A \quad \text{ and } \quad \iota_G \otimes_\mu u_{G,\mu} = (j_{A\rtimes_{\max} G}\circ \iota_G) \cdot (j_{C_\mu^* G }\circ u_{G,\mu})\,.\]
		One immediately sees that $(\iota_A\otimes_\mu 1, \, \iota_G\otimes_\mu u_{G,\mu})$
		is of the form claimed by the lemma. 
		Starting from the formulas and, since the universal covariant representation $(\iota_A,\iota_G)$ is involved, it follows immediately that $(\iota_A\otimes_\mu 1, \, \iota_G\otimes_\mu u_{G,\mu})$ is a nondegenerate
		covariant representation of~$A$. It integrates to a nondegenerate *-representation
		\[(\iota_A \otimes_\mu 1) \rtimes_{\max} (\iota_G \otimes_\mu u_{G,\mu})\colon A\rtimes_{\max} G
		\to \mathcal{M}((A\rtimes_{\max} G)\otimes_\mu C_\mu^* G )\,.\]
		It remains to show that
		$(\iota_A \otimes_\mu 1) \rtimes_{\max} (\iota_G \otimes_\mu u_{G,\mu})$ factors through
		$q_A\colon
		 A\rtimes_{\max} G \to A\rtimes_\mu G$.
		
		To this end, we note that $(A\rtimes_{\max} G, \Ad(\iota_G))$ defines a $G$-C*-algebra whose
		action is unitarily implemented.  Because
		$-\rtimes_\mu G$ is a correspondence crossed product functor by assumption, by Proposition~\ref{CF:extEq} the C*-algebra
		$(A\rtimes_{\max} G)\rtimes_{\mu,\,\Ad(\iota_G)}G$ is *-isomorphic to
		\[(A\rtimes_{\max} G)\rtimes_{\mu,\,\operatorname{id}}G
		= (A\rtimes_{\max} G)\otimes_\mu C^*_\mu G \,.\]
		The *-isomorphism
		$\varphi\colon(A\rtimes_{\max} G)\rtimes_{\mu,\, \Ad(\iota_G)}G \to (A \rtimes_{\max} G) \otimes_\mu C_\mu^* G$
		 is given by
		\begin{align*}
			\varphi(f) \coloneqq (G \to A\rtimes_{\operatorname{max}} G,\, s \mapsto f(s) \iota_G(s))
		\end{align*}
		for $f\in (A\rtimes_{\max} G )\rtimes_{\alg,\, \Ad(\iota_G)}G$.
		Furthermore, $\iota_A\colon A \to \mathcal{M}(A\rtimes_{\max} G)$ is a $G$-equivariant
		generalised *-homomorphism, where $\Ad(\iota_G)$ is the action used on $A\rtimes_{\max} G$, and therefore induces a *-homo\-mor\-phism
		\[\iota_A \rtimes_\mu G \colon A\rtimes_\mu G \to \mathcal{M}((A\rtimes_{\max} G)\rtimes_{\mu,\Ad(\iota_G)}G)\,.\]
		
		We now define
		\[\pi\coloneqq \overline{\varphi}\circ (\iota_A \rtimes_\mu G)\colon A\rtimes_\mu G \to \mathcal{M}((A \rtimes_{\max} G) \otimes_\mu C_\mu^* G)\,.\]
		Then $\pi$ and $(\iota_A \otimes_\mu 1) \rtimes_{\max} (\iota_G \otimes_\mu
		u_{G,\mu})$ coincide on the dense subspace $A\rtimes_{\alg} G$. Hence 
		$(\iota_A \otimes_\mu 1) \rtimes_{\max} (\iota_G \otimes_\mu u_{G,\mu})$ factors through
		$q_A\colon A\rtimes_{\max} G \to A\rtimes_\mu G$.
		
		If $G$ is discrete, then Proposition~\ref{pre:injectiveIndcoRep} states that $\iota_A\rtimes_\mu G$ is injective and hence
		$\pi$ is injective. This implies that $(\iota_A \otimes_\mu 1) \rtimes_\mu (\iota_G \otimes_\mu u_{\mu,G})$
		is injective.
	\end{proof}
	\begin{rem}
The Fell absorption principle for exotic crossed products (Theorem \ref{main_thm_intro_fell} in the introduction) is a direct consequence of the previous lemma.
	\end{rem}
\subsection{Exact correspondence crossed product functors}

\begin{defn}
Let $-\rtimes_\mu G$ be a crossed product functor. It is called \emph{exact} if for every $G$-C*-algebra $A$ and every $G$-invariant ideal $I$ of $A$ the sequence
\[0 \to I\rtimes_\mu G \to A\rtimes_\mu G \to (A/I)\rtimes_\mu G \to 0\]
is exact.
\end{defn}

\begin{cor}\label{cor:exact-trivial}
Let $-\rtimes_\mu G$ be a correspondence crossed product functor for a discrete group $G$.
Then it is exact if and only if it is exact for trivial actions, that is, if the functor $A\mapsto A\otimes_\mu C^*_\mu G$ is exact on the category of C*-algebras.
\end{cor}
\begin{proof}
The forward direction is trivial.

For the converse we use the embedding $A\rtimes_\mu G\into (A\rtimes_{\max} G) \otimes_\mu C^*_\mu G $ provided by Lemma~\ref{lem_strong_fell}. Then the proof is exactly the same as the one for the fact that a group is exact if and only if its reduced group C*-algebra is exact. For convenience we sketch the proof here. Given a $G$-invariant ideal $I\subseteq A$, the exactness of $-\rtimes_{\max} G$ yields a short exact sequence $I\rtimes_{\max} G\into A\rtimes_{\max} G\onto (A/I)\rtimes_{\max} G$. Using the maps provided by Lemma~\ref{lem_strong_fell} for $I$, $A$ and $A/I$, we get a commutative diagram
\[
\xymatrix{
(I\rtimes_{\max} G) \otimes_\mu C^*_\mu G  \ar[r] & (A\rtimes_{\max} G)\otimes_\mu C^*_\mu G \ar[r]^-{\tilde q} & ((A/I)\rtimes_{\max} G) \otimes_\mu C^*_\mu G \\
I\rtimes_\mu G \ar[r]^-{i} \ar[u] & A\rtimes_\mu G \ar[r]^-{q} \ar[u]^-{\delta} & (A/I)\rtimes_\mu G \ar[u]
}
\]
where the vertical arrows are the injective homomorphisms provided by Lemma~\ref{lem_strong_fell} and the top line is exact by assumption. 

The following argument now yields the exactness of the bottom line at $A\rtimes_\mu G$: let $x\in A\rtimes_\mu G$ with $q(x)=0$. Then $\delta(x)\in \ker(\tilde q)=J \coloneqq (I\rtimes_{\max}G)\otimes_\mu C^*_\mu G$. Let $(e_i)$ be an approximate unit for $I\rtimes_\mu G$. Since $\delta$ is nondegenerate (by Lemma~\ref{lem_strong_fell}), $\delta(e_i)$ is an approximate unit for $J$. Since $\delta(x)\in J$ we have  $\delta(xe_i)=\delta(x)\delta(e_i)\to \delta(x)$. But since $\delta$ is isometric, this is equivalent to $xe_i\to x$. Therefore $x\in I\rtimes_\mu G$.

The injectivity of $i$ follows from the ideal property of $\rtimes_\mu$, and surjectivity of $q$ holds, since it is quickly seen that it must have dense image.
\end{proof}
	
\begin{rem}
The full crossed product functor $A\mapsto A\rtimes_{\max} G$ is always an exact correspondence functor. The group $C^*$-algebra of this functor is, by definition, $C^*_{\max}(G)$, the full group $C^*$-algebra of $G$. The other functor of interest for us will be the minimal exact correspondence functor $A\mapsto A\rtimes_\epsilon G$ whose group $C^*$-algebra is denoted by $C^*_\epsilon(G)$. This is, in a precise sense, the smallest crossed product functor which is at the same time exact and strongly Morita compatible (a correspondence functor).
\end{rem}

For the proof of Theorem \ref{thm_low_degree_exotic} we need to extend an exact correspondence
	crossed product functor $-\rtimes_\mu G$ to an exact correspondence crossed product functor 
	$-\rtimes_{\widetilde{\mu}}(G\times\mathbb{Z})$ such that 
	$C^*_{\widetilde{\mu}}(G\times \mathbb{Z}) \cong C_\mu^* G \otimes C^* \mathbb{Z}$. This is the content of the next lemma.
	
\begin{lem}\label{lem_extension_crossed_product}
Let $-\rtimes_\mu G$ be a crossed product functor. Then there is a crossed product functor $-\rtimes_{\widetilde{\mu}}(G\times\mathbb{Z})$ such that $C^*_{\widetilde{\mu}}(G\times \mathbb{Z}) \cong C_\mu^* G \otimes C^* \mathbb{Z}$. Further,
\begin{enumerate}
\item if $-\rtimes_\mu G$ is exact, then $-\rtimes_{\widetilde{\mu}}(G\times\mathbb{Z})$ is exact, and
\item if $-\rtimes_\mu G$ is a correspondence functor, then $-\rtimes_{\widetilde{\mu}}(G\times\mathbb{Z})$ also is one.
\end{enumerate}
\end{lem}

\begin{proof}
	First of all, note that the functor $-\rtimes_\mu G$ induces a functor from $\CstarAlg(G\times \mathbb{Z})$ to $\CstarAlg(\mathbb{Z})$, which we 
	again denote by $-\rtimes_\mu G$, in the following way:
	
	For a $(G\times \mathbb{Z})$-C*-algebra $A$ with action $\alpha\colon G\times \mathbb{Z}\to
	\mathrm{Aut}(A)$ we write $\alpha_{G} = \alpha_{G\times \lbrace 0 \rbrace}$ and 
	$\alpha_{\mathbb{Z}} = \alpha_{\lbrace e \rbrace \times \mathbb{Z}}$. Then
	$\alpha_{\mathbb{Z}}(n)$ is a $G$-equivariant *-homomorphism on $(A,\alpha_G)$ for all 
	$n\in \mathbb{Z}$. Hence it induces a *-homomorphism 
	$\alpha_{\mathbb{Z}}(n) \rtimes_\mu G \colon A\rtimes_{\mu,\, \alpha_G } G \to A\rtimes_{\mu,\, \alpha_G}G$ for all $n\in \mathbb{Z}$. Since $-\rtimes_\mu G$ is a functor,
	$\alpha_{\mathbb{Z}}\rtimes_{\mu} G\colon \mathbb{Z}\to \mathrm{Aut}(A\rtimes_{\mu,\, \alpha_{G}} G),\, n \mapsto \alpha_{\mathbb{Z}}(n)\rtimes_\mu G$ defines a group homomorphism.
	Finally, $(A\rtimes_{\mu,\, \alpha_G} G,\, \alpha_{\mathbb{Z}}\rtimes_\mu G)$ defines
	a $\mathbb{Z}$-C*-algebra.
		
	For a $(G\times \mathbb{Z})$-equivariant *-homomorphism $\varphi\colon A\to B$ 
	between $(G\times \mathbb{Z})$-C*-algebras $A$ and $B$ with actions $\alpha$ and $\beta$ we 
	obtain a $\mathbb{Z}$-equivariant *-homomorphism
	\[\varphi\rtimes_\mu G\colon (A\rtimes_{\mu,\, \alpha_G}G,\, \alpha_\mathbb{Z}\rtimes_\mu G) \to (B\rtimes_{\mu,\, \beta_G}G,\, \beta_\mathbb{Z}\rtimes_\mu G)\,.\]

	Hence we define $-\rtimes_{\widetilde{\mu}}(G\times \mathbb{Z})$ to be $(-\rtimes \mathbb{Z})\circ (-\rtimes_\mu G) \colon \CstarAlg(G\times \mathbb{Z})\to \CstarAlg$. By construction, $C_{\widetilde{\mu}}^*(G\times \mathbb{Z}) = 
 	(\mathbb{C}\rtimes_\mu G) \rtimes \mathbb{Z} \cong C_\mu^* G \otimes C^*\mathbb{Z}$.

 	To prove that $-\rtimes_{\widetilde{\mu}}(G\times \mathbb{Z})$ is a crossed 
 	product functor, let $q\colon -\rtimes_{\max} G \to -\rtimes_\mu G$ and
	$s \colon -\rtimes_\mu G \to -\rtimes_{r} G$ be the natural transformations. They induce
	natural transformations
	\begin{alignat*}{3}	
	\widetilde{q} & \colon & (-\rtimes \mathbb{Z}) \circ (-\rtimes_{\max}G) & \to (-\rtimes \mathbb{Z}) \circ (\rtimes_{\mu} G)\\
	\widetilde{s} & \colon & (-\rtimes \mathbb{Z}) \circ
	(-\rtimes_\mu G) & \to (-\rtimes \mathbb{Z}) \circ (-\rtimes_r G)
	\end{alignat*}
	consisting of quotient maps.
	Furthermore, there exist natural isomorphisms between $-\rtimes_{\max} (G\times \mathbb{Z})$
	and $(-\rtimes \mathbb{Z}) \circ (-\rtimes_{\max} G)$ and 
	between
	 $-\rtimes_{r} (G\times \mathbb{Z})$
	and $(-\rtimes \mathbb{Z}) \circ (-\rtimes_{r} G)$ \cite[Prop.~3.11]{williams}. Hence
	$-\rtimes_{\widetilde{\mu}}(G\times \mathbb{Z})$ is a crossed product functor.

	Note that if $-\rtimes_\mu G$ is exact, then the new functor from $\CstarAlg(G\times \mathbb{Z})$ to $\CstarAlg(\mathbb{Z})$ is again exact. Then 
	$-\rtimes_{\widetilde{\mu}}(G\times \mathbb{Z})$ is, as a composition of exact functors,
 	again exact. 
 	
 	Therefore,
 	it remains to show that $\rtimes_{\widetilde{\mu}}(G\times \mathbb{Z})$ is a
 	correspondence crossed product functor if $-\rtimes_\mu G$ is one. 
	Let $A$ be a $(G\times \mathbb{Z})$-C*-algebra with action $\alpha$ and 
	$p\in \mathcal{M}(A)$ be a $(G\times \mathbb{Z})$-invariant projection. Then $p$ is
	$G$-invariant and since $-\rtimes_\mu G$ is a correspondence crossed product functor,
	the inclusion $pAp \to A$ induces an injective *-homomorphisms $pAp\rtimes_{\mu,\, \alpha_G}G
	\to A\rtimes_{\mu,\, \alpha_G} G$ (see \cite[Thm.~4.9]{bew1}). Because $-\rtimes \mathbb{Z}$ 
	maps injective $\mathbb{Z}$-equivariant *-homomorphisms to injective *-homomorphisms, we finally 
	have that
	$pAp \to A$ induces an injective *-homomorphism $pAp\rtimes_{\widetilde{\mu}}(G\times \mathbb{Z}) \to A\rtimes_{\widetilde{\mu}}(G\times \mathbb{Z})$. Hence $-\rtimes_{\widetilde{\mu}}(G\times \mathbb{Z})$ satisfies the projection property, which implies by 
	\cite[Thm.~4.9]{bew1}
	that $-\rtimes_{\widetilde{\mu}}(G\times \mathbb{Z})$ is a correspondence
	crossed product functor.
\end{proof}

\section{Application to the strong Novikov conjecture}

In this section we will revisit the proof of Theorem~\ref{hanke_schick_intro} in Section~\ref{subsec_original_proof}, and then modify it in Section~\ref{subsec_modifying_proof} such that we conclude the stronger statement claimed in Theorem~\ref{main_thm_intro} in the introduction. Commutativity of the big diagram occurring in the proof of Theorem~\ref{main_thm_intro} is proven in the separate Section~\ref{subsec_commutativity_diag}.

\subsection{The original proof of Hanke and Schick}
\label{subsec_original_proof}

We will start with a more general setup than the one of Theorem~\ref{hanke_schick_intro} in the introduction: we will first follow the exposition given in \cite{hanke_K_area} about elements of infinite $K$-area.

Recall that there exists a natural pairing $K^*(X) \otimes K_*(X) \to \IZ$ between the $K$-theory and $K$-homology of a compact Hausdorff space $X$. It can be described $K\!K$-theoretically as the Kasparov product $K\!K_{-*}(\IC,C(X)) \otimes K\!K_*(C(X),\IC) \to K\!K_0(\IC,\IC) \cong \IZ$ under the identifications 
$K^*(X) \cong K\!K_{-*}(\IC,C(X))$ and $K_*(X) \cong K\!K_*(C(X),\IC)$. Concrete formulas for this pairing may be found in, e.g., \cite[Sec.\ 8.7]{higson_roe}. For any ($\sigma$-unital) C*-algebra $A$, the Kasparov product also induces a pairing
\[
 \langle -,- \rangle \colon KK(\IC,C(X)\otimes A)\otimes KK(C(X),\IC)\to KK(\IC, A)
\]
which is used in the following definition.

\begin{defn}[Hanke~{\cite[Defn.~3.5]{hanke_K_area}}]
Let $M$ be a closed, smooth manifold and let us consider a $K$-homology class $h \in K_0(M)$.

We say that $h$ has \emph{infinite $K$-area}, if there exists a Riemannian metric on $M$  so that the following holds: for each $\varepsilon > 0$ there is a unital C*-algebra $A_\varepsilon$ and a finitely generated Hilbert $A_\varepsilon$-module bundle $E_\varepsilon \to M$ which carries a holonomy representation which is $\varepsilon$-close to the identity
and satisfies
\begin{equation}
\label{eq_infinite_K_area}
\langle [E_\varepsilon], h\rangle \not= 0 \in K_0(A_\varepsilon) \otimes \IQ
\end{equation}
where $[E_\varepsilon] \in \KK(\IC, C(M) \otimes A_\varepsilon)$ is the element represented by $E_\varepsilon \to M$.
\end{defn}
Here the class $[E_{\varepsilon}]\in K_0(C(X)\otimes A_{\varepsilon})\cong KK(\IC,C(X)\otimes A_{\varepsilon})$ is represented  by the finitely generated projective $C(X)\otimes A_{\varepsilon}\,$-module of the sections of $E_{\varepsilon}$.
Also in the previous definition, a holonomy representation $\mathcal{H}$ on $E \to M$ is \emph{$\varepsilon$-close to the identity}, if for each $x \in M$ and each contractible, closed, piecewise smooth loop $\gamma$ based at $x \in M$ the following holds: Let $F\colon D^2 \to M$ be a piecewise smooth map which restricts to $\gamma$ on $S^1 = \partial D^2$. Then
\[
\|\mathcal{H}(\gamma) - \id_{E_x}\|_{\op} \le \varepsilon \cdot A(F(D^2))\,,
\]
where $A(D)$ denotes the area of $F(D^2)$.

\begin{rem}
In the reference \cite{hanke_K_area} the notion \emph{$\varepsilon$-close to the identity at the scale $l$} was used instead the one we use here. But the problem is that with the notion \emph{at scale $l$} the proof of Proposition~3.12 in loc.\ cit.\ does not seem to work. It does work with the version of the notion that we present here. Furthermore, Proposition~3.4 in loc.\ cit.\ still works with this slightly stronger version of this notion of \emph{$\varepsilon$-close to the identity} and hence we are still fine in the low degree setting of Theorem~\ref{thm_low_degree_exotic}.

That the original notion \emph{$\varepsilon$-close to the identity at the scale $l$} is not sufficient was communicated to us by Benedikt Hunger, and the above used new version of this notion by Bernhard Hanke.
\end{rem}

Let $\alpha\colon K_*(M) \to K_*(C^\ast_\mathrm{max} \pi_1(M))$ be the higher index map. We briefly recall now how it is constructed.
Let $\mathcal{M\! F}:= \widetilde{M} \times_{\pi_1(M)} C^*_{\operatorname{max}}\pi_1(M)$ be the Mishchenko--Fomenko bundle. It is a bundle of finitely generated Hilbert $C^*_{\textrm{max}}\pi_1(M)$-modules and hence defines a class
$[\mathcal{M\! F}] \in K_0(C(M)\otimes C_\mathrm{max}^*\pi_1(M)) $.
Then $\alpha \coloneqq \langle [\mathcal{M}\mathcal{F}], - \rangle$ is the index pairing with this class. If $M$ is a closed aspherical manifold, then the strong Novikov conjecture predicts it to be rationally injective.

\begin{thm}[Hanke~{\cite[Thm.~3.9]{hanke_K_area}}]\label{thm_hanke}
Let $M$ be a closed connected smooth manifold and let $h \in K_0(M)$ be of infinite $K$-area.

Then we have
\[\alpha(h) \not= 0 \in K_0(C^\ast_{\mathrm{max}} \pi_1(M)) \otimes \IR\,.\]
\end{thm}

\begin{proof}
We will give a sketch of Hanke's proof since we will need the details of it later. We will provide the definitions of the appearing objects as we go along.

Hanke constructs a commutative diagram
\begin{equation}
\label{eq_diagram_hanke}
\xymatrix{
K_0(M) \ar[rr]^-{\langle [\mathcal{M\! F}],-\rangle}_-{\alpha} \ar[d]_{=} & & K_0(C^\ast_\mathrm{max} \pi_1(M)) \ar[rr]^-{\phi_\ast} & & K_0(Q) \ar[d]_{=}\\
K_0(M) \ar[rr]^-{\langle [V],-\rangle} & & K_0(A) \ar[rr]^-{\psi_\ast} & & K_0(Q)
}
\end{equation}
and shows that $h$ is sent to something non-zero in the lower-right corner. The top-left arrow is the map $\alpha$ represented as twisting by the Mishchenko--Fomenko bundle $\mathcal{M\! F}$ as explained before.
The steps in Hanke's proof are the following:
\begin{enumerate}
\item\label{Point_1} 
Since $h$ has infinite $K$-area we can find, directly by definition, unital C*-algebras $A_{1/k}$ for every $k \in \IN$ 
and finitely generated Hilbert $A_{1/k}\,$-bundles $E_{1/k}$ over $M$ carrying a holonomy representation which is $1/k$-close to the identity. 

In the above Diagram~\ref{eq_diagram_hanke} we have that $A = \prod_{k=1}^\infty A_{1/k}$ (norm bounded sequences) and 
$V \to M$ is a finitely generated Hilbert $A$-module bundle arising from the sequence $(E_{1/k})_k$
via a direct product construction. More precisely, Hanke shows in \cite[Prop.~3.12]{hanke_K_area} that thanks to uniform bounds on the holonomy representations we can choose a cocycle of transition functions for every $E_{1/k}$ with its Lipschitz constants bounded independently of $k$. Hence we have a well defined bundle $V$ whose transition functions are the one for $E_{1/k}$ for all $k$ placed in diagonal form. In other words the $k$-component of $V$ is isomorphic to $E_{1/k}$ as a Hilbert $A_{1/k}$-module bundle. 

It is important to stress that the existence of this construction is ensured by the fact that all the holonomy representations of the component bundles $E_{1/k}$ are uniformly close to the identity. If such a Lipschitz uniformity in the transition functions is not ensured, such a construction cannot be performed as an example in loc.~cit.~shows.

Also an important point in the construction is the fact that we can assume the fiber of $E_{1/k}$ to be isomorphic to $q_k A_{1/k}$ with $q_k$ a projection in $A_{1/k}$ (this is obtained up to tensoring with matrices). It follows that the typical fiber of $V$ is just $qA$ with $q=(q_k)_k \in A$.

\item\label{defn_flat_bundle_W} Let $A^\prime$ be the closed ideal $\bigoplus A_{1/k}$ in $A$ and $Q$ be the quotient C*-algebra $A/A'$ with quotient map $\psi\colon A\to Q$. We define $\psi_k\colon A \to A_{1/k}$ as the projection onto the $k$-th component.

We get a bundle of finitely generated Hilbert $Q$-modules with fiber $\psi({q})Q$, namely $W \coloneqq V\otimes_{\psi} Q$. Thanks to the crucial property on the holonomy representations of the component bundles one proves that $W$ is flat and associated with a unitary holonomy representation $\phi\colon \pi_1(M) \to \operatorname{Hom}_Q(\psi(q)Q,\psi(q)Q) =\psi(q)Q\psi(q).$ Composing with the inclusion $\psi(q)Q\psi(q) \hookrightarrow Q$ and passing to the maximal group C*-algebra we get a morphism
\begin{equation}\label{morphismpi}
C^\ast_{\mathrm{max}} \pi_1(M) \to Q\,,
\end{equation}
see \cite[Prop.\ 3.13]{hanke_K_area}.
This is the morphism $\phi$ in \eqref{eq_diagram_hanke}.

\item\label{step_three} The composition $K_0(M) \xrightarrow{\langle [V],-\rangle} K_0(A) \xrightarrow{(\psi_k)_{\ast}} K_0(A_{1/k})$ sends the element $h$ to $\langle [E_{1/k}],h\rangle \in K_0(A_{1/k})$ which is rationally non-zero by assumption on $h$. Therefore, under the map
\begin{equation}
\label{eq_map_chi}
\chi \colon K_0(A) \to \prod_{k \in \IN} K_0(A_{1/k}), \quad z \mapsto ((\psi_k)_{\ast}z)_{k=1, 2, \ldots}
\end{equation}
the element $z := \langle [V],h\rangle$ is sent to a sequence all of whose components are non-zero.

\item We consider the short exact sequence $0 \to A^\prime \to A \to Q \to 0$ which provides the long exact sequence
\[\cdots \xrightarrow{\partial} K_0(A^\prime) \to K_0(A) \to K_0(Q) \xrightarrow{\partial} \cdots\]
Assuming that $\psi_{\ast}(z) = 0 \in K_0(Q)$ we get a lift of $z$ to $K_0(A^\prime)$. Because $K$-theory commutes with direct sums, we have $K_0(A^\prime) \cong \bigoplus_{k=1}^\infty K_0(A_{1/k})$. But this implies that the sequence $\chi(z) = ((\psi_k)_{\ast}z)_{k=1, 2, \ldots}$ is non-zero for only finitely many $k \in \IN$, which is a contradiction. Hence $\psi_{\ast}(z) \not= 0$, and from the above diagram we therefore get that $\alpha(h) \not= 0$ in $K_0(C^\ast_{\mathrm{max}} \pi_1(M))$.
\item The last step entails checking that the above arguments go through if we tensor everything with $\IR$. This is straightforward.\qedhere
\end{enumerate}
\end{proof}

\subsection{Classes of degree \texorpdfstring{$2$}{2} have infinite \texorpdfstring{$K$}{K}-area}\label{2classesinfinite}

Let $G$ be a finitely presented group.
As already said, we prove Theorem \ref{main_thm_intro} by modifying the proof of  Theorem \ref{thm_hanke}. The Diagram \eqref{eq_diagram_hanke} will be replaced by a bigger one involving a closed manifold $M$, a corresponding bundle $V$ and a quotient algebra $Q$. The bundle $V$ will be in turn defined by the product construction from a collection of algebras $A_{1/k}$ and bundles $E_{1/k}$ as before. 
In this section we recall the procedure of Hanke and Schick \cite{hanke_schick} showing how all these ingredients are created out of  a pair 
\begin{equation}\label{pairgroup}
(h,c) \quad \textrm{with:}\quad h \in K_0(BG)\,, \quad c\in H^2(BG;\mathbb{Z})\quad \textrm{and}\quad \langle c, \ch(h)\rangle \neq 0\,.
\end{equation}

For the commutativity of the new diagram, the very specific form of $A_{1/k}$, of $E_{1/k}$ and in particular of the natural traces that the algebras $A_{1/k}$ possess will be essential.

Let $(h,c)$ be as above. 
By the Baum--Douglas model for $K$-homology, $h$ is represented in terms of a finite-dimensional Hermitian vector bundle $S \to M$ over a closed, connected spin-manifold $M$ equipped with a continuous map $f\colon M \to BG$, i.e. $f_*(S \cap [\slashed{D}_M])=h$. Here $[\slashed{D}_M] \in K_0(M)$ is the canonical $K$-homology class of the Dirac operator $\slashed{D}_M$ of the spin structure of $M$, and $S \cap [\slashed{D}_M]$ is its cap product with the $K$-theory class of $S$.
 Since $G$ is finitely presented, $f$ can be taken to induce isomorphism of fundamental groups\footnote{This is the reason why we restrict in this argument to finitely presented groups $G$.}.
Denote by $\pi \colon \widetilde{M}\longrightarrow M$ the universal cover of $M$. Thanks to the assumptions on $f$, we identify $G$ with the deck group of $\widetilde{M}$.
Hanke and Schick construct 
on the associated Hilbert bundle $\widetilde{M}\times_G \ell^2(G)$, where $G$ acts diagonally (by the right regular representation on $\ell^2(G)$ and deck transformations on $\widetilde{M}$), a family $( \nabla^{1/k})_k$ of connections with curvature going to zero with $k \to \infty$. This is the reason for the Hilbert $Q$-module $W$ being flat (see Point~\ref{defn_flat_bundle_W} of the proof of Theorem~\ref{thm_hanke}). 

More precisely, we have a natural left action $\varphi \mapsto g\cdot \varphi$ of $G$ on the space of forms on $\widetilde{M}$ with values $\operatorname{End}(\ell^2(G))$ and a 
connection form $\eta \in \Omega^1(\smash{\widetilde{M}},i \mathbb{R})$.
Let us recall how $\eta$ is defined: We choose a Hermitian connection on the line bundle $L\to M$ classified by $f^*(c)$. Since the universal cover of $BG$ is contractible, $\pi^*(L)$ is trivializable. After choosing a unitary trivialization, $\eta$ is the connection form of the image of the (induced) connection on $\pi^*(L)$ under it.

Using the $G$--action on the forms on $\smash{\widetilde{M}}$, we can define a natural family of invariant connection forms $(\omega_{1/k})_k$. These are the forms that restrict to $1/k \, (g \cdot \eta)$ on the sub-bundle $\widetilde{M} \times \mathbb{C}\, g$. 
 By $G$-invariance, a corresponding family of connections on $\widetilde{M} \times_G \ell^2(G)$ is well defined. 
Let us choose a reference point $p \in M$ and a point $\widetilde{p} \in \widetilde{M}$ in its fiber. In this way, the fiber of $\widetilde{M}\times_G \ell^2(G)$ at $p$ is identified with $\ell^2(G)$.
We are ready to define the algebras: for every $k$, the algebra $A_{1/k}$ is defined by the norm closure inside $\IB(\ell^2(G))$ of all the operators which arise by parallel translation for $\nabla^{1/k}$ along  piecewise smooth loops. All these algebras are then naturally represented on the same Hilbert space and come with natural traces: the vector states
\[\tau_{1/k}(\vartheta) \coloneqq \langle \vartheta(e),e\rangle \text{ for } \vartheta \in A_{1/k}\,.\]
Here $e \in \ell^2(G)$ is the characteristic function of the identity element $e$ of $G$ and $\langle -, -\rangle$ the inner product of $\ell^2(G)$.
The construction has the following property which will be important for us later and says that the traces are akin to the trace $\tau_e$ on $C_r^*G$.

\begin{Property}\label{prop_traces_supported_on_e}
Let $\phi_{\gamma} \in A_{1/k}$ be the parallel translation map associated to a loop $\gamma$ which represents an element in $\pi_1(M,p)$. Then $\tau_{1/k}(\phi_{\gamma})=0$ if $\gamma$ is not trivial in $\pi_1(M,p)$.
\end{Property}
\begin{proof}
This is also mentioned in the beginning of the proof of  \cite[Lem.~2.2]{hanke_schick}. 
 Let $\gamma$ a loop based on $p$; to compute $\tau_{1/k} (\phi_{\gamma})$ we lift $\gamma$ to a path $\widetilde{\gamma}$ in $\widetilde{M}$ and we compute the parallel translation along $\widetilde{\gamma}$ with respect to the connection associated to $\omega_{1/k}$. We call this operator of parallel translation on the covering $\phi_{\widetilde{\gamma}}$.
  Now assume the class of $\gamma$ in the fundamental group is $g\neq e$, then $\widetilde{\gamma}$ is not a loop and its endpoint is exactly $\widetilde{p} \cdot g$. Since the connection forms preserve the sub-bundles $\mathbb{C} \, g$ 
the vector $\phi_{\gamma} e$ is represented in $\widetilde{M} \times \ell^2(G)$ by the couple $(\widetilde{p} \cdot g, \lambda e)$ for some number $\lambda$. The corresponding inner product computing the trace is $\lambda \langle e, {g^{-1}}\rangle =0$.
\end{proof}
We now construct the bundles $E_{1/k} \to M$. For ease of notation, call $\zeta \coloneqq \widetilde{M}\times_{G}\ell^2(G)$. Then $E_{1/k}$ is the bundle whose fiber at $x$ is the norm closure inside $\IB(\zeta_p,\zeta_x)$ of all the operators which are parallel transport isomorphisms of $\nabla^{1/k}$ along smooth paths joining $p$ and $x$. The $A_{1/k}$-module structure is clear and given by right composition. One can check that every $E_{1/k}$ is locally trivial and equipped with a natural $A_{1/k}$-linear connection induced by $\nabla^{1/k}$.

By construction, for every $k$ the pairing
$\langle [E_{1/k}], S \cap [\slashed{D}_M] \rangle$ is non-trivial, showing that $S \cap [\slashed{D}_M]$ has infinite $K$-area, see \eqref{eq_infinite_K_area}. This is seen by using the natural traces $\tau_{1/k}$ that induce real valued functionals $\tau_{1/k} \colon K_0(A_{1/k}) \to \mathbb{R}$ satisfying
$$\tau_{1/k}(\langle [E_{1/k}], S \cap [\slashed{D}_M] \rangle)\neq 0\,.$$

Summing up: for a finitely presented group $G$ and a pair $(h,c)$  as in \eqref{pairgroup}, we construct the sequences $A_{1/k}$ and $E_{1/k}$ needed to show that the $K$-homology class $S \cap [\slashed{D}_M]$ of $M$ has infinite $K$-area. In particular, for fixed $k$ the non-triviality of the pairing $\langle [E_{1/k}],S \cap [\slashed{D}_M]\rangle$ is witnessed by a very particular trace on $A_{1/k}$.

Finally, note that $f_*(S \cap [\slashed{D}_M])=h$ for a continuous map $f$ inducing an isomorphism on fundamental groups $\pi_1(M) \cong \pi_1(BG)$. The latter implies
\[
\alpha(S \cap [\slashed{D}_M]) = \mu^{\mathrm{BC}}(f_*(S \cap [\slashed{D}_M]))\,,
\]
where $\alpha$ is the higher index map from Diagram~\ref{eq_diagram_hanke} and $\mu^{\mathrm{BC}} \colon K_*(BG) \to K_*(C^*_{\max} G)$ is the analytic assembly map.

\subsection{Incorporating Fell's absorption principle}
\label{subsec_modifying_proof}

The proof of Theorem~\ref{thm_low_degree_exotic} works with any exact correspondence crossed product functor. The properties of such functors, that we will need in the proof, are the following ones:

\begin{properties}\label{properties_crossed_product}
Fix a discrete group $G$. Let $-\rtimes_\mu G$ be an exact correspondence crossed product functor.

Then it has the following properties:
\begin{enumerate}
\item\label{properties_group_alg_between_max_red} The corresponding group C*-algebra $C^*_\mu G := \IC \rtimes_\mu G$ is a completion of $\IC G$ and the identity on $\IC G$ extends to surjective *-homomorphisms $C^*_{\max} G \to C^*_\mu G \to C^*_r G$.

This property holds by definition of crossed product functors.
\item\label{properties_between_max_red} For any C*-algebra $A$ with the trivial $G$-action, the crossed product $A \rtimes_\mu G$ is a C*-completion of the algebraic tensor product $A \odot C^*_\mu G$ and the identity map on $A \odot C^*_\mu G$ extends to surjective *-homomorphisms
\[A \otimes_{\max} C^*_\mu G \to A \rtimes_\mu G \to A \otimes_{\mathrm{min}} C^*_\mu G\,.\]
Because of this we denote $A \rtimes_\mu G =: A \otimes_\mu C^*_\mu G$.

This property is exactly Proposition~\ref{prop_is_tensor_product}.
\item\label{properties_exact} The functor $- \otimes_\mu C^*_\mu G$ is exact, i.e., for every exact sequence $0 \to I \to A \to Q \to 0$ of $C^*$-algebras with the trivial $G$-action we get an exact sequence
\[0 \to I \otimes_\mu C^*_\mu G \to A \otimes_\mu C^*_\mu G \to Q \otimes_\mu C^*_\mu G \to 0\,.\]

This is true since we assume $-\rtimes_\mu \, G$ to be exact. Note that it is actually equivalent to requiring that $- \otimes_\mu C^*_{\mu}G$ is exact by Corollary~\ref{cor:exact-trivial}.
\item\label{properties_coprod} The coproduct $\Delta\colon \IC G \to \IC G \odot \IC G$, $\sum a_g g \mapsto \sum a_g (g \otimes g)$ extends continuously to a $^*$-homomorphism
\[\Delta\colon C_\mu^* G \to C_{\max}^* G \otimes_\mu C_\mu^* G\,.\]

This follows from Lemma~\ref{lem_strong_fell}. Note that because $G$ is a discrete group, $C_{\max}^*G$ and $C_\mu^* G$ are unital C*-algebras. Hence $C_{\max}^*G \otimes_\mu C_\mu^*G$ is also unital and so its multiplier algebra coincides with it.

Note that this property in particular requires that the coproduct extends to a map $C^*_\mu G\to C^*_{\max}G\otimes_{\min} C^*_\mu G$. And this is equivalent to require that the dual space $(C^*_\mu G)^\prime$ is (isomorphic to) an ideal in the Fourier-Stieltjes algebra  $B(G)=(C^*_{\max}G)^\prime$, see \cite[Cor.~3.13]{KLQ-Exotic}.
\item\label{properties_direct_sum} The canonical map
\[\bigoplus_{k \in \IN} \Big(A_k \otimes_\mu C^*_\mu G\Big) \to \Big(\bigoplus_{k \in \IN} A_k \Big) \otimes_\mu C^*_\mu G\]
is an isomorphism.

This property is exactly Lemma~\ref{lem-directsums}.
\qedhere
\end{enumerate}
\end{properties}

We can now generalize the result of Hanke--Schick \cite{hanke_schick}. Let $G$ be a finitely presented group\footnote{Hanke and Schick first prove their theorem for finitely presented groups, and then use that every discrete group is a filtered colimit of finitely presented groups to generalize their result to all discrete groups. But in our situation, since the $K$-theory of the reduced group C*-algebra $C^*_r(G)$ -- and of $C^*_{\varepsilon}(G)$ -- is not known to be functorial for arbitrary group homomorphisms, we can not carry out the last step generalizing to all discrete groups.} and denote by $\Lambda^\ast(G) \subset H^\ast(BG;\IQ)$ the subring generated by $H^{\le 2}(BG;\IQ)$. Recall further the homological Chern character $\ch\colon K_*(BG)\to H_*(BG;\IQ)$; a definition of it using the Baum--Douglas model for $K$-homology may be found in \cite[§11]{baum_douglas}. Note that we are also using the Baum--Douglas model for $K$-homology in the proof below.

\begin{thm}\label{thm_low_degree_exotic}
Let $h \in K_*(BG)$ be such that there is $c \in \Lambda^\ast(G)$ with $\langle c,\ch(h)\rangle \not= 0$, and let $-\rtimes_\mu G$ be an exact correspondence crossed product functor.
Then $h$ is not mapped to zero under the assembly map $K_*(BG) \to K_*(C^\ast_\mu G) \otimes \IR$.
\end{thm}

\begin{rem}
Theorem~\ref{thm_low_degree_exotic} applies in particular for the group $C^*$-algebra $C^*_\epsilon(G)$ and therefore also for every other exotic group $C^*$-algebra above it, that is, any completion of the group algebra $\IC G$ for a $C^*$-norm above the norm giving $C^*_\epsilon(G)$.
\end{rem}

\begin{rem}\label{rem_suspension_argument}
Let us explain the suspension trick used in the odd-dimensional case of the  proof of Theorem~\ref{thm_low_degree_exotic}.

Let us denote by $z \in K_1(S^1)$ a generator. We have exterior products
\begin{alignat*}{3}	
	-\boxtimes z & \colon & K_*(BG) & \to K_{*+1}(B(G \times \IZ))\,,\\
	-\boxtimes \alpha(z) & \colon & K_*(C^*_\mu G) & \to K_{*+1}(C^*_{\widetilde\mu}(G \times \IZ))\,,
\end{alignat*}
where we have used Lemma~\ref{lem_extension_crossed_product} to define the exterior product for the group C*-algebras, and both exterior products are injective by the respective Kuenneth theorems.

If $h \in K_1(BG)$ satisfies the assumption of the Theorem~\ref{thm_low_degree_exotic}, then $h \boxtimes z$ also satisfies it (since $c \boxtimes \varphi_z$ is an element of $\Lambda^*(G \times \IZ)$ if $c \in \Lambda^\ast(G)$, where $\varphi_z \in H^1(S^1)$ is a generator with $\varphi_z(\ch z) = 1$).

Therefore, if we can show that $h \boxtimes z$ is not mapped to zero under the assembly map $K_0(B(G \times \IZ)) \to K_0(C^*_{\widetilde\mu}(G \times \IZ)) \otimes \IR$, then $h$ will not be mapped to zero under the assembly map $K_1(BG) \to K_1(C^\ast_\mu G) \otimes \IR$.
\end{rem}

\begin{proof}[Proof of Theorem~\ref{thm_low_degree_exotic}]
The proof is an adaptation of the proof of \cite[Thm.~4.1]{hanke_K_area}, which is itself an elaboration on the proof given in \cite{hanke_schick}.

By the suspension argument explained in Remark~\ref{rem_suspension_argument} we can assume that $h \in K_0(BG)$. For simplicity we also assume that $c \in H^2(BG;\IZ)$. The general case $c \in \Lambda^*(G)$ reduces to the former case as described at the end of Section 2 of \cite{hanke_schick}.
We thus have a couple $(h,c)$ as in \eqref{pairgroup}. Indeed, we first perform the construction of $M$, of $(A_{1/k})_k$, the bundles $(E_{1/k})_k$ and traces $\tau_{1/k}$ explained in  section \ref{2classesinfinite}.
We apply to these ingredients the infinite bundle construction we summarised in the proof of Theorem~\ref{thm_hanke} and we get the following diagram, which is a modified and expanded version of  Diagram~\eqref{eq_diagram_hanke}. We will explain the arrows occurring in it further below.
\begin{equation}
\label{eq_diagram_hanke_modified}
\mathclap{
\xymatrix{
K_0(M) \ar[r]^-{\langle [\mathcal{M\! F}],-\rangle}_-{\alpha_\mu} \ar[d]_{=} & K_0(C^\ast_\mu G) \ar[r]^-{\Delta_\ast} & K_0(C^*_\mathrm{max} G \otimes_\mu C^*_\mu G) \ar[r]^-{(\phi \otimes \mathrm{id})_\ast} & K_0(Q \otimes_\mu C^*_\mu G) \ar@{-->}@/^1pc/[ddddl] \\
K_0(M) \ar[r]^-{\langle [V],-\rangle} & K_0(A) \ar[r]^-{\psi_\ast} \ar[d]^{(\iota_A)_*} & K_0(Q) \ar[d] &\\
& K_0(A \otimes_\mu C^*_\mu G) \ar[r]^-{(\psi \otimes \mathrm{id})_\ast} \ar[d]^{\prod (\psi_k \otimes \mathrm{id})_\ast} & K_0(Q \otimes_\mu C^*_\mu G) \ar@{-->}[dd]^{} &\\
& \prod K_0(A_{1/k} \otimes_\mu C^*_\mu G) \ar[d]^{\prod (\tau_{1/k} \otimes \tau_e)} &&&\\
& \prod \IR \ar[r] & \prod \IR \big/ \bigoplus_{\textrm{alg}} \IR &&
}
}
\end{equation}

The first arrow in the top line is the composition of the map $\alpha\colon K_0(M) \to K_0(C^\ast_\mathrm{max} G)$ from Diagram~\eqref{eq_diagram_hanke} with the map induced from the $^*$-homomorphism $C^*_\mathrm{max} G \to C^*_\mu G$ given by Property~\ref{properties_crossed_product}.\ref{properties_group_alg_between_max_red}. Note that $\alpha_\mu$ factors as the composition of $K_0(M) \to K_0(BG)$ with the higher index map $K_0(BG) \to K_0(C^\ast_\mu G)$, where the first of these maps is induced by a choice of classifying map $M \to BG$.

The map $\phi$ and the morphism $(\phi \otimes \textrm{id})_*$ with it arise from the holonomy representation of a flat bundle as in \eqref{morphismpi}. This is defined on $C^*_{\mathrm{max}}\pi_1(M)$. Since $f$ induces an isomorphism of  fundamental groups, we get a map defined on $C^*_{\mathrm{max}}G$.

The second arrow in the top line is induced by the coproduct whose existence is guaranteed by Property~\ref{properties_crossed_product}.\ref{properties_coprod}. The bottom vertical map on the left
\[\prod (\tau_{1/k} \otimes \tau_e)\colon \prod K_0(A_{1/k} \otimes_\mu C^*_\mu G) \to \prod \IR\]
is induced from the maps $A_{1/k} \otimes_\mu C^*_\mu G \to A_{1/k} \otimes_{\mathrm{min}} C^*_\mu G$ (which exist by Property~\ref{properties_crossed_product}.\ref{properties_between_max_red}) composed with the tensor products $\tau_{1/k} \otimes \tau_e$ of the given traces $\tau_{1/k}$ on the algebras $A_{1/k}$ with the canonical trace $\tau_e$ on $C_r^* G$. By $\bigoplus_{\textrm{alg}} \IR$ in the lower right corner we mean the algebraic direct sum, i.e. sequences with only finitely many non-zero entries.

The dashed arrows will be constructed at the end of this proof, and the commutativity of the diagram will be shown in Corollary~\ref{cor_commutativity_huge_diagram} below.

Given all the above, the claim of this theorem follows quickly: the left path from $K_0(M)$ to $\prod \IR$, applied to the element $h$, results in a sequence all of whose entries are non-zero due to the assumptions of this theorem (compare to Step~\ref{step_three} of the proof of Theorem~\ref{thm_hanke}). Hence it stays non-zero if we map it further to $\prod \IR \big/ \bigoplus_{\textrm{alg}} \IR$. By commutativity of the diagram this means that $\alpha_{\mu}(h) \in K_0(C^\ast_\mu G)$ is non-zero. Tensoring the diagram with~$\IR$ we still obtain the same conclusion. Hence $\alpha(h) \not= 0 \in K_0(C^\ast_\mu G) \otimes \IR$.\footnote{Note that there is no concrete reason here to tensor with $\IR$. We could instead also tensor with $\IQ$.}

It remains to construct the dashed arrows in  Diagram~\eqref{eq_diagram_hanke_modified}, i.e., the map
\[K_0(Q \otimes_\mu C^*_\mu G) \to \prod \IR \big/ {\bigoplus}_{\textrm{alg}} \IR\,.\]
An element $x \in K_0(Q \otimes_\mu C^*_\mu G)$ is represented by a difference $x = [p] - [q]$ of projections $p$ and $q$ in matrices over $Q \otimes_\mu C^*_\mu G$. We will now explain how to evaluate a single projection like $p$ to something in $\prod \IR \big/ \bigoplus_{\textrm{alg}} \IR$. 

By Property~\ref{properties_crossed_product}.\ref{properties_exact} we have an exact sequence
\begin{equation}\label{exactseq}
0 \to A^\prime \otimes_\mu C^*_\mu G \to A \otimes_\mu C^*_\mu G \to Q \otimes_\mu C^*_\mu G \to 0,
\end{equation}
and we will show:
\begin{itemize}
\item for any small $\varepsilon$, say $\varepsilon <\sfrac{1}{8}$, we find a $K \in \mathbb{N}$ and a "lift" $p^\prime_{>K} \in \prod_{>K} (A_{1/k}\otimes_{\mu}C^*_{\mu}G)$ which is an $\varepsilon$-projection and can be evaluated suitably by $\prod_{> K}(\tau_{1/k}\otimes \tau_e)$. The result is in $\prod \IR \big/ \bigoplus_{\textrm{alg}} \IR$, and
\item this is independent on the choices.
\end{itemize}
We give the details in the following.

From the exactness of \eqref{exactseq} we can lift $p$ to a self-adjoint matrix $\tilde p$ over $A \otimes_\mu C^*_\mu G$ which is a projection modulo matrices over $A^\prime \otimes_\mu C^*_\mu G$. We can apply now the map $\prod (\psi_k \otimes \mathrm{id})$ to map $\tilde p$ to a self-adjoint matrix $p^\prime$ over $\prod (A_{1/k} \otimes_\mu C^*_\mu G)$.
Because of the Property~\ref{properties_crossed_product}.\ref{properties_direct_sum} we have $\prod (\psi_k\otimes \mathrm{id})(A'\otimes_\mu C_\mu^* G)\subseteq \bigoplus_k (A_{1/k}\otimes_\mu C_\mu^* G)$, thus $p^\prime$ will be a projection modulo matrices over $\bigoplus (A_{1/k} \otimes_\mu C^*_\mu G)$.
Hence, fixing an $\varepsilon < \sfrac{1}{8}$, there will be $K \in \IN$, such that $p^\prime_{>K}$ is an $\varepsilon$-projection in $\prod_{>K} (A_{1/k} \otimes_\mu C^*_\mu G)$. So, if $\mathrm{proj}(-)$ denotes the characteristic function of the interval $[1-\sfrac{1}{4},1+\sfrac{1}{4}]$, we get an honest projection $\mathrm{proj}(p^\prime_{>K})$ in $\prod_{>K} (A_{1/k} \otimes_\mu C^*_\mu G)$ which can be evaluated by $\prod_{>K} (\tau_{1/k} \otimes \tau_e)$ to a value in $\prod \IR \big/ \bigoplus_{\textrm{alg}} \IR$. It remains to show that this is well-defined. If we have a second lift $p^{\prime\prime}_{>K^\prime}$, then $\mathrm{proj}(p^\prime_{>K})$ and $\mathrm{proj}(p^{\prime\prime}_{>K^\prime})$ will be $(\sfrac{1}{2} + \sfrac{1}{10})$-close in $\prod_{>K^{\prime\prime}} (A_{1/k} \otimes_\mu C^*_\mu G)$ for some large $K^{\prime\prime}$ and hence they will be unitarily equivalent.\footnote{It is a general fact about C*-algebras that projections $P$, $Q$ with $\|P-Q\|<1$ are unitarily equivalent.}
\end{proof}

\subsection{Commutativity of the main diagram}
\label{subsec_commutativity_diag}

Let us explain why commutativity of Diagram~\eqref{eq_diagram_hanke_modified} is non-trivial. We denote by
\begin{itemize}
\item $\IC G$ the complex group ring of $G$,
\item $\Delta\colon \IC G \to \IC G \odot \IC G$ the coproduct $\sum a_g g \mapsto \sum a_g (g \otimes g)$,
\item $\iota\colon \IC G \to \IC G \odot \IC G$ the inclusion $\sum a_g g \mapsto \sum a_g (g \otimes e)$, where $e$ is the identity element of the group $G$,
\item $\tau\colon \IC G \to \IC$ any trace on $\IC G$, and
\item $\tau_e\colon \IC G \to \IC$ the canonical trace $\sum a_g g \mapsto a_e$ on $\IC G$.
\end{itemize}
Now we consider the diagram
\begin{equation}
\xymatrix{
 \IC G \ar[r]^-{\Delta} \ar[d]_-{\iota} & \IC G \odot \IC G \ar[d]^{\tau \otimes \tau_e}\\
 \IC G \odot \IC G \ar[r]_-{\tau \otimes \tau_e} & \IC
}
\label{eq_nontriviality}
\end{equation}
We have
\[((\tau \otimes \tau_e)\circ \Delta)\Big(\sum a_g g\Big) = a_e \cdot \tau(e)\]
and for the other composition in the diagram we have
\[((\tau \otimes \tau_e)\circ\iota)\Big(\sum a_g g\Big) = \sum a_g \tau(g)\,.\]
Hence Diagram~\eqref{eq_nontriviality} only commutes in the case that $\tau$ is a multiple of $\tau_e$.

Now Diagram~\eqref{eq_nontriviality} is similar to Diagram~\eqref{eq_diagram_hanke_modified} in the sense that the composition of the top arrows of Diagram~\eqref{eq_diagram_hanke_modified} with the dashed arrow is similar to the composition of the top and right vertical arrow in Diagram~\eqref{eq_nontriviality}, and the composition of the horizontal arrows in the second row of Diagram~\eqref{eq_diagram_hanke_modified} with the dashed arrow is similar in flavour to the composition of the left vertical and lower arrow in Diagram~\eqref{eq_nontriviality}. So we expect commutativity of Diagram~\eqref{eq_diagram_hanke_modified} only if the traces $\tau_{1/k}$ on the algebras $A_{1/k}$ used in the definition of the dashed arrows in Diagram~\eqref{eq_diagram_hanke_modified} are of similar kind as the canonical evaluation trace on the identity element in the group C*-algebras.
But this is exactly Property \ref{prop_traces_supported_on_e} as we discussed in Section \ref{2classesinfinite}.


Before proving that Diagram~\eqref{eq_diagram_hanke_modified} commutes we discuss some basic facts about almost projections and $K$-theory.

Let $A$ be a unital  $C^*$-algebra and keep fixed a small $\varepsilon >0$ for the entire following discussion. Let us explain how self-adjoint \emph{almost idempotents} over $A$ define $K$-theory classes (see \cite[Sec.~2.2]{XieYu}).\footnote{The self-adjointness is actually not strictly necessary for this. We have incorporated it since in our situation it can always be arranged.} An almost idempotent is an element $p \in A$ which satisfies $\| p^2-p\|<\varepsilon$. 
Let $p$ be a self-adjoint almost idempotent in $M_n(A)$; then
there are disjoint open sets $U,V \subset \mathbb{C}$ with disjoint closure
separating $0$ and $1$ in its spectrum:
$$\operatorname{spec}(p) \subset U \cup V\,, \quad 0 \in U\,, \quad 1 \in V\,. $$ 
Choose a real-valued function $h$ on $\IC$ with $h(\xi)=0$ on $U$ and $h(\xi)=1$ on $V$.
Then the holomorphic  functional calculus, integrating the function $h$
over a contour $\mathcal{C}$ surrounding $\operatorname{spec}(z)$ inside $U \cup V$, produces an honest projection and we define the $K$-theory class of $p$ to be
$$[p]:=\Big{[}\frac{1}{2\pi i}\int_{\mathcal{C}} h(\xi)(\xi-p)^{-1} d\xi\Big{]}\,.$$
Notice that we do not actually need the holomorphic functional calculus because we are discussing the construction using selfadjoints. However, since this works just as well for quasi-idempotents, we continue using it.
Denote by $V_{\varepsilon}(A)$ the space of self-adjoint almost idempotents of matrices over $A$; by considering formal differences in the procedure before, we have constructed a surjection $V_{\varepsilon}(A) \times V_{\varepsilon}(A) \to K_0(A)$. Every ${}^\ast$-homorphism $f\colon A \to B$ of (unital) $C^*$-algebras $A$ and $B$ is contractive, therefore it restricts to a map $V_{\varepsilon}(A) \to V_{\varepsilon}(B)$ which induces a map $f_*\colon K_0(A) \to K_0(B)$.  

Let now $T\colon A \to Z$ be a positive tracial map with $Z$ a commutative $C^*$-algebra; it induces a map $T_*\colon K_0(A) \to Z$.
Note that $T_*$ maps to the self-adjoint elements in $Z$.
Let us compute $T_*([p])$ for the class of a self-adjoint almost idempotent. Continuity implies  
$$T_*([p])= \frac{1}{2\pi i}\int_{\mathcal{C}}h(\xi)T ((\xi-p)^{-1})d\xi\,.$$
To keep track of these classes which are defined by self-adjoint almost idempotents it can be useful to introduce the notation 
$$ \Big{[}\int y\Big{]}:= \Big{[}\frac{1}{2\pi i} \int_{\mathcal{C}} h(\xi)(\xi - y)^{-1} d\xi \Big{]}\,.$$
In this way, under a morphism $F$ we have $F_*[\int y]=[\int F(y)]$.

With this in mind we reconstruct the dashed map $\tau_*\colon K_0(Q\otimes_{\mu}C^*_{\mu}G)\to \prod \IR  / \bigoplus_{\textrm{alg}}\IR$ using self-adjoint almost idempotents.  Since general classes are formal differences, we explain as before the construction for a single self-adjoint almost idempotent element $p \in M_n ( Q \otimes_{\mu}C^*_{\mu}G)$ with $\|p^2-p\|<\varepsilon$.
Let $p^\prime \in M_n(A\otimes_{\mu}C^*_{\mu}G)$ be any lift in the sequence \eqref{exactseq}. By considering $(p^\prime + {p^\prime}^*)/2$ we can assume that $p^\prime$ is self-adjoint. By definition of the quotient norm on $M_n(Q \otimes_{\mu}C^*_{\mu}G)$ we have
\[
\inf_{y \in  M_n(A' \otimes_{\mu}C^*_{\mu}G)}\| {p^\prime}^2 - {p^\prime} -y\|= \|p^2-p\| <\varepsilon \,,
\]
where the norm on the left hand side is the one on $M_n(A \otimes_{\mu}C^*_{\mu}G)$ and the norm on the right hand side the one on $M_n(Q \otimes_{\mu}C^*_{\mu}G)$.
We can thus find a $y_0 \in  M_n(A' \otimes_{\mu}C^*_{\mu}G)$ with $\| {p^\prime}^2 - {p^\prime} -y_0\| < \varepsilon$; without loss of generality we may assume that $y_0$ is self-adjoint (by passing to $(y_0 + y_0^*)/2$). 
We continue to denote by $\psi_k \otimes \operatorname{id}$ the extension of the morphism to matrix algebras. Of course, we also exchange the tensor product with $M_n$ and $\prod$. Then
we look at the image $$\big(\prod_{k}\psi_k \otimes \operatorname{id} \big)(p^\prime)= (q_k)_k \in \prod_k M_n(A_{1/k}\otimes_{\mu}C^*_{\mu}G)\,.$$
Recall that $A'$ is the direct sum of the $A_{1/k}$, and let be $K$ such that $\|(\psi_k \otimes \operatorname{id} ) y_0 \|< \varepsilon$ for $k >K$. It follows that $$q_{>K}:=(q_k)_{>K} \in \prod_{k>K} M_n(A_{1/k}\otimes_{\mu}C^*_{\mu}G)$$ is a self-adjoint almost idempotent with $\|q_{>K}^2- q_{>K}\|<2\varepsilon$.
Extend as customary the map $\prod_k \tau_k \otimes \tau_e$ to matrices over $\prod_k A_{1/k} \otimes_{\mu}C^*_{\mu}G$ by tensoring with the matrix trace. We continue to denote it with the same symbol.

Now it remains to apply the functional calculus to define
\[\tau_* ([p]) \coloneqq \Big( \prod_{k>K} \tau_k \otimes \tau_e\Big)_* \Big{[}\int q_{>K}\Big{]} \, \,\operatorname{mod}\,\, \oplus_{\textrm{alg}}\, \mathbb{R}\,.\]
This procedure defines a map 
	$\tau_*\colon K_0(Q\otimes_{\mu}C^*_{\mu}G)\longrightarrow \prod \IR  / \bigoplus_{\textrm{alg}}\IR$ which coincides with the dashed arrow constructed in the proof of Theorem \ref{thm_low_degree_exotic}.
	
\begin{prop}
We have a commutative diagram
\begin{equation*}
\mathclap{
\xymatrix{
			K_0(C_{\mathrm{max}}^* G) \ar[rr]^-{(\iota_{C^*_{\mathrm{max}}G})_*} \ar[d] &&
				K_0(C_{\mathrm{max}}^*G \otimes_\mu C_\mu^* G) \ar[rr]^-{(\phi \otimes_\mu \mathrm{id})_*} &&
				K_0(Q\otimes_\mu C_\mu^* G) \ar[r]^-{\tau_*}  
				& \prod \IR /\bigoplus_{\alg} \IR\\
			K_0(C_\mu^* G) \ar[rr]^-{\Delta_*} &&
				K_0(C_{\mathrm{max}}^* G \otimes_\mu C_\mu^* G) \ar[rr]^-{(\phi\otimes_\mu \mathrm{id})_*} &&
				K_0(Q\otimes_\mu C_\mu^*G) \ar[ru]^-{\tau_*}
		}
}
	\end{equation*}	
\end{prop}
\begin{proof}
Start with a class $[x] \in K_0(C^*_{\textrm{max}}G)$ represented 	by a true projection in $M_n(C^*_{\operatorname{max}}G)$. By density we can find a self-adjoint $y \in M_n(\mathbb{C}G)$ which approximates $x$. 
We write it as a finite sum 
$$y= \sum a_g \, g\,, \quad a_g \in M_n(\mathbb{C})\,.$$
Then $y$ is a self-adjoint almost idempotent. It follows that $y$ represents $[x]$ in the sense of the discussion before. Let us move $[x]$ to $K_0(Q \otimes_{\mu}C^*_{\mu}G)$ according to the up route and the down route in the above diagram. We find two elements:
$$y_1=\Big{[}\int \sum_{g} a_g \phi(g) \otimes e \Big{]} \quad \textrm{and}\quad y_2=\Big{[}\int \sum_{g} a_g \phi(g) \otimes g \Big{]}\,.$$
We have to show that applying $\tau_*$ we get the same result. 
Remembering that $A$ is the product of all the holonomy algebras along loops, for these two elements we have a class of preferred lifts in $A \otimes_{\mu}C^*_{\mu}G$.
Indeed
  for any element $g \in G$ we choose a smooth loop $\gamma(g)$ representing $g$ in the fundamental group.   Denote  with $\operatorname{Hol}(\gamma(g))=(\operatorname{Hol}^{1/k}(\gamma(g)))_{k}\in A$ the collection of the parallel translations along this loop with respect to $\nabla^{1/k}$,
then we have two lifts of $y_1$ and $y_2$. These are
$$z_1=\sum a_g \operatorname{Hol}(\gamma(g)) \otimes e\,, \quad  {\textrm{and}}\quad  z_2=\sum a_g \operatorname{Hol}(\gamma(g)) \otimes g$$ in $M_n \otimes (A \otimes_{\mu}C^*_{\mu}G)\cong M_n  (A \otimes_{\mu}C^*_{\mu}G)$. We can assume that $z_i$ with $i=1,2$ are: 
\begin{itemize}
\item  self-adjoint (for if not, just take $(z_i + z_i^*)/2$ --- they lift in the same way, because $y_1$ and $y_2$ are self-adjoint),
\item 	almost idempotent (because if not, we can cut away a finite number of components of the $((\phi_g)^{1/k})_k$ as in the discussion before).
\end{itemize}
 Now given the integral formula for $\tau_*$ and the use of the functional calculus the proof will be complete if we manage to show that for any polynomial with real coefficients $f$ we have
$$\big(\prod_k \tau_k \otimes \tau_e \big)\big(\prod_k \psi_k \otimes \operatorname{id}\big) f(z_1) = \big(\prod_k \tau_k \otimes \tau_e \big)\big(\prod_k \psi_k \otimes \operatorname{id}\big) f(z_2).$$
Let us check it for any power: we have
$$z_1^m= \sum_{(g_1,...,g_m)\in G^m }a_{g_1}\cdots a_{g_m}\operatorname{Hol}(\gamma(g_1)\cdots \gamma(g_m)) \otimes e $$ and 
$$z_2^m= \sum_{(g_1,...,g_m)\in G^m }a_{g_1}\cdots a_{g_m}\operatorname{Hol}(\gamma(g_1)\cdots \gamma(g_m)) \otimes g_1 \cdots g_m\,.$$ 
We apply our maps $\prod \tau_k \otimes \tau_e$ (recall that these maps have been extended to matrices) and then look at the $k$th-component of the result:
\begin{align}\nonumber
\Big{(}\big(\prod_k \tau_k \otimes \tau_e \big)\big (\prod_k \psi_k &\otimes \operatorname{id}\big)z_1^m\Big{)}_k =\\  \label{sumtrace}
=& 
\sum_{(g_1,...,g_m) \in G^m} \operatorname{tr}_{M_n}(a_{g_1}\cdots a_{g_m} )\tau_k \operatorname{Hol}^{1/k}(\gamma(g_1) \cdots \gamma(g_k))
\end{align} 
The element $\gamma(g_1) \cdots  \gamma(g_m)$ represents the product $g_1 \cdots g_m \in \pi_1(M,p)$; it follows by Property~\ref{prop_traces_supported_on_e} that the sum in \eqref{sumtrace} is just performed on the elements $(g_1,...,g_m)$ such that $g_1 \cdots g_m=e$. This is to say that 
$$\big(\prod_k \tau_k \otimes \tau_e \big)\big(\prod_k \psi_k \otimes \operatorname{id}\big) z_1^m= \big(\prod_k \tau_k \otimes \tau_e \big)\big(\prod_k \psi_k \otimes \operatorname{id}\big) z_2^m$$
finishing this proof.
\end{proof}

From the above proposition we can conclude the sought corollary:
\begin{cor}\label{cor_commutativity_huge_diagram}
Diagram~\eqref{eq_diagram_hanke_modified} commutes.
\end{cor}


\bibliography{./Bibliography_SNC_low_degree_exotic}

\providecommand{\bysame}{\leavevmode\hbox to3em{\hrulefill}\thinspace}
\providecommand{\MR}{\relax\ifhmode\unskip\space\fi MR }
\providecommand{\MRhref}[2]{%
  \href{http://www.ams.org/mathscinet-getitem?mr=#1}{#2}
}
\providecommand{\href}[2]{#2}
\begin{thebibliography}{BEW18b}

\bibitem[AAS]{AAS3}
P.~Antonini, S.~Azzali, and G.~Skandalis, \emph{{The Baum--Connes conjecture
  localised at the unit element of a discrete group}}, to appear in Compositio
  Math., \href{https://arxiv.org/abs/1807.05892}{arXiv:1807.05892}.

\bibitem[AAS16]{MR3419768}
\bysame, \emph{Bivariant {$K$}-theory with {$\Bbb{R}/\Bbb{Z}$}-coefficients and
  rho classes of unitary representations}, J. Funct. Anal. \textbf{270} (2016),
  no.~1, 447--481.

\bibitem[BD82]{baum_douglas}
P.~Baum and R.~G. Douglas, \emph{{$K$-homology and index theory}}, {Operator
  Algebras and Applications} (R.~Kadison, ed.), Proc. Symp. Pure Math.,
  vol.~38, Amer. Math. Soc., 1982, pp.~117--173.

\bibitem[BEW17]{bew2}
A.~Buss, S.~Echterhoff, and R.~Willett, \emph{Exotic crossed products},
  Operator algebras and applications---the {A}bel {S}ymposium 2015, Abel Symp.,
  vol.~12, Springer, 2017, pp.~67--114.

\bibitem[BEW18a]{bew1}
\bysame, \emph{{Exotic crossed products and the Baum---Connes conjecture}}, J.
  reine angew. Math. (Crelles Journal) \textbf{740} (2018), 111--159.

\bibitem[BEW18b]{bew3}
\bysame, \emph{{The minimal exact crossed product}}, Doc. Math. \textbf{23}
  (2018), 2043--2077.

\bibitem[BGW16]{BGW}
P.~Baum, E.~Guentner, and R.~Willett, \emph{Expanders, exact crossed products,
  and the {B}aum-{C}onnes conjecture}, Ann.\ K-Theory \textbf{1} (2016), no.~2,
  155--208.

\bibitem[CGM93]{MR1204787}
A.~Connes, M.~Gromov, and H.~Moscovici, \emph{Group cohomology with {L}ipschitz
  control and higher signatures}, Geom. Funct. Anal. \textbf{3} (1993), no.~1,
  1--78.

\bibitem[Han11]{hanke_K_area}
B.~Hanke, \emph{{Positive scalar curvature, $K$-area and essentialness}},
  {Global Differential Geometry} (Ch. B{\"a}r, J.~Lohkamp, and M.~Schwarz,
  eds.), Springer Proceedings in Mathematics, Springer-Verlag, 2011,
  pp.~275--302.

\bibitem[HLS02]{MR1911663}
N.~Higson, V.~Lafforgue, and G.~Skandalis, \emph{Counterexamples to the
  {B}aum-{C}onnes conjecture}, Geom.\ Funct.\ Anal. \textbf{12} (2002), no.~2,
  330--354.

\bibitem[HR00]{higson_roe}
N.~Higson and J.~Roe, \emph{{A}nalytic {K}-{H}omology}, Oxford University
  Press, New York, 2000.

\bibitem[HS08]{hanke_schick}
B.~Hanke and T.~Schick, \emph{{The strong Novikov conjecture for low degree
  cohomology}}, Geom. Dedicata \textbf{135} (2008), 119--127.

\bibitem[KLQ13]{KLQ-Exotic}
S.~Kaliszewski, M.~Landstad, and J.~Quigg, \emph{{Exotic group $C^*$-algebras
  in noncommutative duality}}, New York J. Math. \textbf{19} (2013), 689--711.

\bibitem[Mat03]{MR1998926}
V.~Mathai, \emph{The {N}ovikov conjecture for low degree cohomology classes},
  Geom. Dedicata \textbf{99} (2003), 1--15.

\bibitem[RW98]{raeburn_williams}
I.~Raeburn and D.~P. Williams, \emph{{Morita Equivalence and Continuous-Trace
  C*-Algebras}}, Mathematical Surveys and Monographs, vol.~60, AMS, 1998.

\bibitem[Wil07]{williams}
D.~P. Williams, \emph{{Crossed Products of C*-Algebras}}, Mathematical Surveys
  and Monographs, vol. 134, AMS, 2007.

\bibitem[XY14]{XieYu}
Zh. Xie and G.~Yu, \emph{{A relative higher index theorem, diffeomorphisms and
  positive scalar curvature}}, Adv. Math. \textbf{250} (2014), 35--73.

\end{thebibliography}
\bibliographystyle{amsalpha}

\end{document}